\documentclass[reqno]{amsart}
 \usepackage{amsbsy,amssymb,amscd,amsfonts,latexsym,amstext,delarray,
 amsmath,graphicx, color}
 \usepackage[dvips]{epsfig}
\usepackage{hyperref}
\input xypic

\newtheorem{thm}{Theorem}[section]
\newtheorem{prop}[thm]{Proposition}
\newtheorem{cor}[thm]{Corollary}
\newtheorem{lem}[thm]{Lemma}

\newtheorem{defn}[thm]{Definition}
\newtheorem{rem}[thm]{Remark}

\def\C{{\mathbb C}}

\def\R{{\mathbb R}}
\def\Z{{\mathbb Z}}

\def\T{{\mathbb T}}

\def\cA{{\mathcal A}}

\def\cH{{\mathcal H}}

\def\cK{{\mathcal K}}
\def\cL{{\mathcal L}}

\def\cR{{\mathcal R}}

\def\Tr{{\rm Trace}}

\def\sin{{{\rm sin}}}
\def\cos{{{\rm cos}}}

\def\qqq{\,,\quad~\forall}

\def\ries{\Sigma}
\def\lap{\triangle}
\def\lapa{\triangle}

\def\modu{\Delta}
\def\pert{\triangle_\vp}
\def\perta{\triangle_\vp}
\def\pertb{\triangle^{(1, 0)}_\vp}
\def\pertc{\triangle^{(0, 1)}_\vp}
\def\pertd{\triangle^{(0, 1)}_\epsilon}
\def\D{\triangle}
\def\dvp{{\bar D_\varphi}}
\def\ttd{{\tilde{\triangle}_\ve}}

\def\fH{{\mathfrak H}}

\def\a{\alpha}
\def\b{\beta}
\def\d{\delta}
\def\g{\gamma}

\def\Om{\Omega}

\def\ve{\varepsilon}
\def\vp{\varphi}

\def\cutint{{\int \!\!\!\!\!\! -}}

\newcommand{\ie}{{i.e.\/}\ }
\newcommand{\eg}{{e.g.\/}\ }
\newcommand{\cf}{{cf.\/}\ }
\newcommand{\comp}{{comp.\/}\ }

\def\text{\hbox}

\def\Det{{\rm Det}}

\def\Grad{{\rm Grad}}
\def\grad{{\rm grad}}

\def\Ker{{\rm Ker}}

\def\Tr{{\rm Tr}}

\def\T{{\rm T}}
\def\r{{\rm r}}

\def\lmod{\nabla}
\def\newval{\log \left(4 \pi^2 \,  |\eta(\tau)|^4 \right)}

\newcommand{\ch}{\cosh}
\newcommand{\sh}{\sinh}

\newcommand{\nil}[1]{}
\title[Modular curvature]{Modular curvature  for
noncommutative two-tori}

\begin{document}

\author[Connes]{Alain Connes}
\author[Moscovici]{Henri Moscovici}
\address{A.~Connes: Coll\`ege de France \\
3, rue d'Ulm \\ Paris, F-75005 France\\
I.H.E.S. and Vanderbilt University} \email{alain\@@connes.org}
\address{H.~Moscovici:
Department of mathematics, The Ohio State University, Columbus, OH 43210, USA}
\email{henri@math.ohio-state.edu}

\thanks{The work of the first named author was partially
 supported by the National Science Foundation
 award no. DMS-0652164}

\thanks{The work of the second named author was partially
 supported by the National Science Foundation
 award no. DMS-0969672}

\begin{abstract}
 
In this paper we investigate the curvature of conformal deformations
by noncommutative Weyl factors of a flat metric
on a noncommutative 2-torus, by analyzing in the framework of spectral triples
functionals 
associated to perturbed Dolbeault operators. 
The analogue of Gaussian curvature turns out to be a sum of
two functions in the modular operator corresponding to the
non-tracial weight defined by the conformal factor, applied to expressions 
involving the derivatives of the same factor. The first is a generating
function for the Bernoulli numbers and is
applied to the noncommutative Laplacian of the conformal factor, while the
second is a two-variable function and is applied to a quadratic form
in the first derivatives of the factor. Further outcomes of the paper include
a variational proof of the Gauss-Bonnet theorem for noncommutative 2-tori, 
the modular analogue of Polyakov's conformal anomaly formula for 
regularized determinants of Laplacians,
a conceptual understanding of the modular curvature as 
gradient of the Ray-Singer analytic torsion, 
and the proof using operator positivity that the scale invariant version
of the latter assumes its extreme value only at the flat metric.
\end{abstract}

\maketitle

\section*{Introduction}

In noncommutative geometry the paradigm of a geometric space is given in spectral terms,
by a Hilbert space $\cH$ in which both the algebra $\cA$ of coordinates and the
analogue of
the inverse line element $ds^{-1}$ are represented, the latter being embodied by
 an unbounded self-adjoint operator  $D$ which plays the role of the Dirac operator.
 The local geometric invariants such as the Riemannian curvature are extracted from 
 the functionals defined by the coefficients of heat kernel expansion 
$$
\Tr(a e^{-tD^2})\, \sim_{t \searrow 0} \, \sum_{n \geq 0} {\rm a}_n(a,D^2)t^{\frac{-d+n}{2}} \, , \quad a \in \cA ,
$$
where $d$ is the dimension of the geometry. Equivalently, 
one may consider special values of the corresponding zeta functions.
 Thus, it is the high frequency behavior of the spectrum of $D$ coupled with the action of the algebra $\cA$ in $\cH$ which detects the local curvature of the geometry.

\medskip

In this paper we implement the Riemannian aspect of this program in great depth
on an archetypal example, that of the noncommutative two torus $\T^2_\theta$, whose differential geometry as well as pseudo-differential operator calculus 
were first developed in \cite{C}. To obtain a curved geometry from the flat one defined in \cite{C}, one introduces (\cf \cite{cc}, \cite{Paula}) a noncommutative Weyl conformal factor (or dilaton), 
which changes the metric by modifying the noncommutative volume form 
while keeping the same conformal structure. 
Both notions of volume form and of conformal structure are well understood in the general case (\cf \cite[\S VI]{Co-book}). We recall in \S \ref{sectprel} how one obtains the modified Dirac operator for the curved geometry obtained from a flat one by modifying the volume form.
 \medskip
 
  The starting point is the computation of the value at $s=0$ of the zeta function $\Tr(a|D|^{-2s})$ for the $2$-dimensional
curved geometry associated to the dilaton $h$, or equivalently of the coefficient
${\rm a}_2(a,D^2)$ of the heat expansion. This computation was
initiated in the late 1980's (\cf \cite{cc}), and the specific result which proves the analogue of the Gauss--Bonnet formula was published in \cite{Paula}. It was subsequently extended in \cite{FK} 
to the case of arbitrary values of the complex modulus $\tau$ (set to $\tau=i$ in
\cite{Paula}). In both these papers only the total integral of the curvature was needed, and this allowed one to make simplifications under the trace which are no longer possible when
$a \neq 1$, \ie when one wants to fully compute the local expression for the functional
$a \in \cA \mapsto {\rm a}_2(a,D^2)$. 
 \medskip
 
While the original computation of \cite{cc} was done entirely
by hand, the technical obstacles encountered when dealing with
 the local computation were overcome by means of the general Rearrangement 
Lemma of \S \ref{sectrearang}, and the assistance of the computer. The latter is not
indispensable, its main role being to facilitate and achieve in a safe way the routine task of
collecting together the large number (around one thousand) of terms
 which arise when applying the generalized pseudo-differential calculus 
and the main algebraic lemma. 
The complete calculation of ${\rm a}_2(a,D^2)$ was actually performed
in 2009 and announced at several conferences
(Oberwolfach 2009 and Vanderbilt 2011), as well as
by internet posting (with some typographical errors). 
The same computation was
 independently done by F. Fathizadeh and M. Khalkhali in
\cite{FK1}, and gave further confirmation to our result.
\medskip

The main additional input of the present paper stems from the fact that we succeeded to express 
in terms of a closed formula the
Ray-Singer log-determinant of $D^2$, issue which was left open in \cite{cc}.
The gradient of the log-determinant functional, or equivalently of the scale invariant version 
of it (cf. \cite{OPS}), yields in turn a local curvature formula,
which arises as a sum of two terms, each involving a 
function in the modular operator, of one and respectively two variables.
Computing the gradient in two different ways leads to the proof of
a deep internal consistency relation between these two distinct constituents, and
at the same time elucidates the meaning of the intricate two operator-variable function.
 
\medskip

We now briefly outline the contents of this paper,
starting with the description of the local curvature functionals determined by
the value at zero of the zeta functions affiliated with the modular spectral triples
describing the curved geometry of noncommutative $2$-tori.
As in the case of the standard torus viewed as a complex curve,
the total Laplacian associated to such a spectral triple splits into two
components, one $\perta$ on functions and the other $\pertc$  on $(0,1)$-forms, the two
operators being isospectral outside zero. The
corresponding curvature formulas involve
second order (outer) derivatives of the Weyl factor, and as a new and
crucial ingredient they involve
 the modular operator $\modu$ of the non-tracial weight $\vp(a)=\vp_0(a e^{-h})$ associated to the dilaton $h$. For $\perta$ the result is of the form
\begin{equation}\label{a2term}
    {\rm a}_2(a,\perta)=-\frac{\pi}{2\tau_2}\vp_0(a\left(K_0(\nabla)(\lapa(h))+\frac 12 H_0(\nabla_1,\nabla_2)(\square_\Re (h)\right) ,
\end{equation}
where $\nabla=\log \modu$ is the inner derivation implemented  by $-h$,
$$
 \lapa(h)=
  \delta_1^2(h)+2 \Re(\tau)\delta_1\delta_2(h)+|\tau|^2 \delta_2^2(h) ,
 $$
 $\square_\Re$ is the Dirichlet quadratic form
 $$
 \square_\Re (\ell) :=
(\delta_1(\ell))^2+ \Re(\tau)\left(\delta_1(\ell)\delta_2(\ell)+\delta_2(\ell)\delta_1(\ell)\right)+|\tau|^2 (\delta_2(\ell))^2\,,
$$
and $\nabla_i, \, i=1, 2$, signifies that $\nabla$ is acting on the $i$th factor.
The operators $K_0(\nabla)$ and $H_0(\nabla_1,\nabla_2)$ are new ingredients, whose
 occurrence is a vivid manifestation of the genuinely non-unimodular nature of
 the conformal geometry of the noncommutative $2$-torus.
 The functions $K_0(u)$ and $H_0(u,v)$ by which the modular derivatives act
  seem at first of a rather formidable nature, and of course beg for a
  conceptual understanding. Their expressions, arising from the computation, are as follows:
  \begin{equation}\label{k0funct}
    K_0(s)=\frac{-2+s\, {\rm coth}\left(\frac{s}{2}\right) }{s\, \sinh\left(\frac{s}{2}\right)} \, ,
\end{equation}
\begin{align}\label{basicquadra1pre}
\begin{split}
&\text{and} \, \qquad \qquad \qquad \qquad \qquad \, H_0(s,t)= \\
&\frac{t (s+t) \ch(s)-s (s+t) \ch(t)+(s-t) (s+t+\sinh(s)+\sinh(t)-\sinh(s+t))}{s t (s+t) \sinh\left(\frac{s}{2}\right) \sinh\left(\frac{t}{2}\right) \sinh\left(\frac{s+t}{2}\right)^2}.
\end{split}
\end{align}
One of our new results consists in giving an abstract proof of a
 functional relation between the functions $K_0$ and $H_0$. More precisely, denoting
$$
\tilde{K}_0(s)\, = \, 4\frac{\sinh(s/2)}{s} K_0(s) \quad \text{and} \quad
\tilde{H}_0(s,t)\, = \, 4\frac{\sinh((s+t)/2)}{s+t} H_0(s,t),
$$
we establish by an a priori argument the identity
\begin{equation}\label{transforulepre}
  - \frac 12 \tilde H_0(s_1,s_2)=\frac{\tilde K_0(s_2)-\tilde K_0(s_1)}{s_1+s_2}+\frac{\tilde K_0(s_1+s_2)-\tilde K_0(s_2)}{s_1}-\frac{\tilde K_0(s_1+s_2)-\tilde K_0(s_1)}{s_2}
\end{equation}
The function $\tilde K_0$ is (up to the factor $\frac 18$) the generating function of the Bernoulli numbers, \ie one has
\begin{equation}\label{bernou1}
\frac 18\tilde K_0(u)=\sum_1^\infty \frac{B_{2n}}{(2n)!}u^{2n-2}\,.
\end{equation}
Another main result consists in obtaining the following closed formula
for the Ray-Singer determinant:
\begin{equation}\label{raysinger}
\log \Det^\prime (\perta) = \log \vp (1)+\newval
+\frac{\pi}{8\tau_2}  \vp_0\left(
\tilde K_0(\nabla_1 )(\square_\Re(h))\right)
\end{equation}
 The a priori proof of the functional relation \eqref{basicquadra1pre} is based on the computation of the gradient of the Ray-Singer determinant in two different ways. Using the left hand side of \eqref{raysinger} one obtains a formula involving ${\rm a}_2(a, \perta)$,
while using the right hand side of \eqref{raysinger} gives a general expression  as shown in Theorem \ref{directgrad} of \S \ref{functrel}.

  \medskip

 As a third fundamental result of this paper, we establish
  the analogue of the  classical result which asserts that in every conformal class the maximum
value of the determinant  of the Laplacian
for metrics of a fixed area is uniquely attained at the constant curvature
metric. This is the content of Theorem \ref{thmmain3}, whose
proof relies on the positivity of the function $\tilde K_0$.
By \eqref{bernou1}, $\tilde K_0$ is a generating function for Bernoulli numbers, known to
play a prominent role in the theory of characteristic classes of deformations,
where it is used as a formal power series. It is quite striking that in the present
context of a conformal (but not formal) deformation, $\tilde K_0$
 appears no longer merely as a formal series but as an actual function, whose
 {\em positivity} plays a key role.

 \medskip

In marked contrast to the ordinary torus, for which
${\rm a}_2(a,\perta)$ and ${\rm a}_2(a,\pertc)$ are both
 constant multiples of the scalar (or Gaussian) curvature,
the local curvature expressions associated to the zeta functions of the two partial Laplacians
differ substantially.
The function $H_1(s,t)$ of two variables involved in the expression of ${\rm a}_2(a,\pertc)$ is
 related to $H_0(s,t)$ in a simple fashion,
 but a new term appears, in the form of an operator
 $S(\nabla_1,\nabla_2)$ applied to the skew quadratic form
\begin{equation}\label{skewpre}
\square_\Im (\ell) :=
i\,  \Im (\tau)\left(\delta_1(\ell)\delta_2(\ell)-\delta_2(\ell)\delta_1(\ell)\right)  , \quad \ell = 2h ;
\end{equation}
It could be useful to find a
fully conceptual understanding of the meaning of this term.

 \smallskip

Being isospectral outside zero, both partial Laplacians have the
same Ray-Singer determinant. This gives rise to a single log-determinant functional,
which represents in fact the analytic torsion of the underlying conformal structure. 
By analogy with the classical case, its
gradient provides the appropriate notion of scalar curvature, and
the corresponding evolution equation for the metric yields
the natural analogue of Ricci flow. A different version of
the latter has been proposed in \cite{BhuM}.

\bigskip

\tableofcontents

\section{Modular spectral triples for noncommutative $2$-tori}\label{sectprel}

The preliminary material gathered in this section is essentially borrowed from \cite{cc}
in order to provide the necessary background for the present paper.
It also serves as a first illustration of the
distinctly non-unimodular feature of  the conformal geometry of
noncommutative $2$-tori, which in particular validates the treatment of twisted
spectral triples~\cite{cm} as basic geometric structures.

\subsection{Inner twisting in the even case} \hfill\medskip

The modular spectral triples considered below can be understood as special cases of the following general construction. Let us start from an ordinary spectral triple $(\cA,\cH,D)$
which we assume to be even (and we let $\gamma$ be the grading operator).  Using the direct sum decomposition $\cH=\cH^+\oplus \cH^-$ the action of the algebra $\cA$, the grading operator and the operator $D$ take the form
\begin{equation}\label{matrix}
    a\mapsto \left(
               \begin{array}{cc}
                 a & 0 \\
                 0 & a \\
               \end{array}
             \right)\,, \ \ \gamma=\left(
                                     \begin{array}{cc}
                                       1 & 0 \\
                                       0 & -1 \\
                                     \end{array}
                                   \right)\,, \ \
                                   D=\left(
                                       \begin{array}{cc}
                                         0 & T^* \\
                                         T & 0 \\
                                       \end{array}
                                     \right)
\end{equation}
were $T$ is an unbounded operator from its domain in $\cH^+$ to $\cH^-$ and $T^*$ is its adjoint. Let now $k\in \cA$ be a positive invertible element. Since the commutator $[D,k]$ is bounded the multiplication by $k$ preserves the domain of $T$ and the following operator is self-adjoint
\begin{equation}\label{dk}
    D_{(k,\gamma)}=\left(
          \begin{array}{cc}
            0 & k T^*  \\
            Tk & 0 \\
          \end{array}
        \right)
\end{equation}
We use the notion of modular (or twisted) spectral triple in the sense of Definition 3.1 of \cite{cm}. Let us show that the perturbation $D_{(k,\gamma)}$ of $D$ defines a twisted spectral triple on $\cA$ with respect to the inner automorphism $\sigma$.
\begin{lem} \label{tiwst1}
Let  $\sigma(a)=kak^{-1}$ be the (non-unitary) inner automorphism of $\cA$ associated to $k$.
The triple $(\cA,\cH,D)$ is a $\sigma$-twisted spectral triple.
\end{lem}
\proof
 Let us compute the twisted commutator $D_{(k,\gamma)} a-\sigma(a)D_{(k,\gamma)}$. One has
$$
D_{(k,\gamma)} a-\sigma(a)D_{(k,\gamma)}=\left(
          \begin{array}{cc}
            0 & kT^*  a\\
            Tka & 0 \\
          \end{array}
        \right)-\left(
          \begin{array}{cc}
            0 & kak^{-1}k T^*  \\
           kak^{-1} Tk & 0 \\
          \end{array}
        \right)
$$
 The upper right element of the matrix gives
$$
kT^*  a-kak^{-1} k T^*=k[T^*,a],
$$
which is bounded since $[D,a]$ is bounded as well as $k$.
 The lower left element of this matrix gives $$
 Tka-ka k^{-1} Tk=T b k-bT k=[T,b]k,\ \ b=\sigma(a),
  $$
  which is also bounded since $b=\sigma(a)\in\cA$.\endproof

  \begin{rem}\label{pertrem}{\rm To display the dependence on the grading $\gamma$ one can
  use the following formula for the perturbation $D_{(k,\gamma)}$ of $D$
  \begin{equation}\label{pertremform}
    D_{(k,\gamma)}=k^{ E}D k^{E}\,, \ \ E=\frac{1+\gamma}{2}.
  \end{equation}
  }\end{rem}

 We shall now explain why it is this simple twisting procedure which is appearing naturally when one introduces a Weyl factor (dilaton) in the geometry of the noncommutative torus. We still need another general notion of transposed spectral triple.

 \subsection{Transposed spectral triple} \hfill\medskip

Given a Hilbert space $\cH$ let $\bar \cH$ be the dual vector space. The transposition $T\mapsto T^t$ gives an antiisomorphism
\begin{equation}\label{antiiso}
    \cL(\cH)\to \cL(\bar \cH)^{\rm op}, \ T\mapsto T^t
\end{equation}
where $\cL(\bar \cH)^{\rm op}$ is the opposite algebra of $\cL(\bar \cH)$.
Thus one can associate to any
 spectral triple $(\cA,\cH,D)$ the transposed spectral triple as follows.
\begin{prop} \label{tiwst2} Let $(\cA,\cH,D)$ be a $\sigma$-twisted spectral triple. Let $\cA^{\rm op}$ be the opposite algebra and $D^t$ the transposed of the unbounded operator $D$. Let $\sigma'$ be the automorphism of $\cA^{\rm op}$ given by
\begin{equation}\label{sigmaprime}
   \sigma'(a^{\rm op})=(\sigma^{-1}(a))^{\rm op}
\end{equation}
Then the action of $\cA^{\rm op}$ in $\bar \cH$ transposed of the action of $\cA$ in $\cH$ defines a $\sigma'$-twisted spectral triple
 \begin{equation}\label{transposedtrip}
   (\cA^{\rm op}, \bar \cH,D^t)
 \end{equation}
 \end{prop}
 \proof The boundedness of the twisted commutators $Da-\sigma(a)D$ implies the boundedness of the twisted commutators
 $$
 D^t a^t-(\sigma^{-1}(a))^t D^t=-\left( D \sigma^{-1}(a)-aD\right)^t.
 $$
 \endproof
 Note that one can identify the dual vector space $\bar \cH$ with the complex conjugate of $\cH$ by the antilinear isometry $J_\cH$
 \begin{equation}\label{jch}
    J_\cH(\eta)(\xi)=\langle \xi,\eta\rangle\qqq \xi ,\eta \in \cH.
 \end{equation}
 One then has the relation
 \begin{equation}\label{reltom}
    T^t=J_\cH T^* J_\cH^{-1}\qqq T\in  \cL(\cH).
 \end{equation}
 \begin{defn} \label{transposetrip}
Given a modular spectral triple $(\cA,\cH,D)$ the transposed modular spectral triple is given by  \eqref{transposedtrip}.
 \end{defn}

\subsection{Notations for $\T^2_\theta$} \hfill\medskip

Let us fix our notations for the noncommutative torus $\T^2_\theta$.
We let $\theta$ be an irrational real number  and
consider the (uniquely determined) $C^*$-algebra $A_{\theta} \equiv C^0(\T^2_\theta)$ generated by
two unitaries
\begin{equation*}
 U^* = U^{-1} \, , \qquad \ V^* = V^{-1} \, ,
\end{equation*}
which satisfy the multiplicative commutation relation
\begin{equation*}
VU = e^{2\pi i \theta} \, UV \, .
\end{equation*}
The $2$-dimensional torus $\T^2 = \left( \R/2\pi\Z\right)^2$
acts on $A_{\theta}$ via the 2-parameter group of automorphisms $\{ \alpha_\r \}$,
$\r \in \R^2$, determined by
\begin{equation*}
\alpha_\r (U^n \, V^m)=e^{i(r_1n + r_2m)}U^n \, V^m \, , \qquad
\r = (r_1, r_2) \in \R^2 \, .
\end{equation*}
We denote by $A_{\theta}^{\infty} \equiv C^\infty (\T^2_\theta)$ the subalgebra of
smooth elements for this action,
\ie consisting of those $x \in A_{\theta}$ such that the mapping
\begin{equation*}
\r \in \R^2 \, \mapsto \, \alpha_\r  (x) \in A_{\theta}
\end{equation*}
is smooth. Expressed in terms of the coefficients of the element $a\in A_{\theta}$,
 $$
 a=\sum_{(n,m) \in \Z^2} a (n,m)U^nV^m \, ,
 $$
 the smoothness condition amounts to their
 rapid decay, \ie the requirement that the sequences
 $\{ \vert n \vert^p \, \vert m \vert^q \, \vert a (n,m) \vert \}_{(n,m) \in \Z^2}$ be bounded for any
  $p, q > 0$.

The basic derivations representing the infinitesimal generators to the above group of automorphisms are given by  the defining relations,
\begin{equation} \label{defder}
\begin{split}
\delta_1 (U) \, &=   \, U \, , \quad \delta_1 (V) = 0 \, , \\
\delta_2 (U) \, &= \, 0 \, , \quad \delta_2 (V) =   \, V \, ;
\end{split}
\end{equation}
they are the counterparts of the differential operators $ \frac 1i\partial / \partial x$,
$ \frac 1i\partial / \partial y$ acting on $C^\infty (\T^2)$, and behave similarly
with respect to the $\ast$-involution:
\begin{equation} \label{invol}
\delta_j (a^\ast) =   \, - \delta_j (a)^\ast\, , \quad j = 1, 2  \qquad
\text{for all} \quad  a \in A_{\theta}^{\infty} \, .
\end{equation}

As $\theta$ was chosen irrational, there is a unique trace $\vp_0$ on $A_{\theta}$,
determined by the orthogonality properties
\begin{equation}\label{trace}
\vp_0 (U^n \, V^m) = 0 \quad \mbox{if} \quad (n,m) \ne (0,0) \, , \quad \mbox{and} \quad \vp_0 (1) = 1 \, ,
\end{equation}
and we denote by ${\cH_0}$  the Hilbert space obtained from $A_{\theta}$ by completing
 with respect to the associated inner product
\begin{equation}\label{inprod}
\langle a,b \rangle = \vp_0 (b^* a) \, , \qquad a,b \in A_{\theta} \, .
\end{equation}
By construction the Hilbert space ${\cH_0}$ is a bimodule over $A_{\theta}$ with
\begin{equation}\label{bim}
    a . \xi . b:= a\xi b \qqq a,b \in \cA, \ \xi \in \cH_0
\end{equation}
and the trace property of $\vp_0$ ensures that the right action of $\cA$ is unitary.

The derivations $\delta_1 , \delta_2$, viewed
as unbounded operators on ${\cH_0}$,
have unique self-adjoint extensions,
\begin{equation} \label{sadelta}
\delta_j^\ast \, =   \, \delta_j  \, , \quad j = 1, 2  \, .
\end{equation}
Furthermore, they obviously obey the integration-by-parts rule
\begin{equation}\label{int-by-parts}
\vp_0(a \delta_j(b)) + \vp_0 (\delta_j(a) b) \, = \, 0 \, , \qquad a,b \in A_{\theta}^\infty \, .
\end{equation}

\subsection{Conformal structures on $\T^2_\theta$} \hfill\medskip

The conformal structures on the classical torus are best parameterized by a complex number $\tau\in \C$, $\Im(\tau)>0$ modulo the natural action of $PSL(2,\Z)$ by homographic transformations. To $\tau$ one associates the lattice $\Gamma=\Z+\tau\Z\subset \C$ and the quotient complex structure on $\T^2\sim\C/\Gamma$. The natural isomorphism of the $2$-dimensional torus $\T^2 = \left( \R/2\pi\Z\right)^2$ (with real coordinates $(x,y)$ as above) with $\C/\Gamma$ is given by
\begin{equation}\label{isotor}
    (x,y)\in \left( \R/2\pi\Z\right)^2\mapsto Z= \frac{1}{2\pi}(x+y\tau)\in \C/\Gamma
\end{equation}
One thus gets
\begin{equation}\label{transp}
     \left(
       \begin{array}{c}
         dZ \\
         d\bar Z \\
       \end{array}
     \right)=\frac{1}{2\pi}\left(
               \begin{array}{cc}
                 1 & \tau \\
                 1 & \bar\tau \\
               \end{array}
             \right)
             \left(
               \begin{array}{c}
                 dx \\
                 dy \\
               \end{array}
             \right)
\end{equation}
This gives $\partial_Z$ and $\partial_{\bar Z}$ as linear expressions in $\partial_x$ and $\partial_y$ and one finds up to the overall factor $\lambda=\frac{2 \pi  \bar{\tau }}{-\tau +\bar{\tau }}$ that
\begin{equation}\label{dz}
    \partial_Z=\partial_x -\frac{1}{\bar{\tau }}\partial_y
\end{equation}
Since replacing the modulus $\tau$ by $-\frac{1}{\tau }$ does not affect the complex structure this allows us to transfer the translation invariant
complex structures of $\T^2$ to $\T^2_\theta$.
Throughout this paper we fix a
complex number $\tau \in \C$ with $\Im (\tau) > 0$ and consider the
associated translation invariant complex structure, defined by the pair of derivations
\begin{equation}\label{deltas}
\d = \delta_1 + \bar{\tau} \delta_2 \, , \qquad \d^\ast =\delta_1 + \tau \delta_2 \, ;
\end{equation}
representing the counterparts of the differential operators
$ \frac 1i \left(\partial / \partial x + \bar{\tau} \partial / \partial y\right)$, and
$ \frac 1i \left(\partial / \partial x + \tau \partial / \partial y\right)$ acting on $C^\infty (\T^2)$. Our conventions differ slightly from \cite{cc} in which the only case $\tau=i$ was covered but we prefer to follow the usual convention for the general case.

As explained in \cite[\S VI. 2]{Co-book}, the conformal (or equivalently, complex) structures
on a Riemann surface can be recast as solutions of a variational problem, for Polyakov action
functionals, involving positive currents in the sense of Lelong that represent
the fundamental class. Since Lelong positivity has a 
natural reformulation in terms of positivity in Hochschild
cohomology, the same type of construction can be extended to noncommutative spaces
with fundamental class. 

In particular (cf.  \cite[\S VI. 3]{Co-book}), for  $A_{\theta}^{\infty}$ 
the information on the conformal structure corresponding to the modulus $\tau$
is encapsulated in the positive Hochschild $2$-cocycle  
\begin{equation}\label{poshoch}
   \phi(a,b,c) \, =\, -\vp_0(a\, \d (b) \, \d^\ast(c)) , \quad
   a, b, c \in A_{\theta}^{\infty} \, ,
\end{equation}
which belongs to the intersection of the positive cone $Z_+^2 (A_{\theta}^{\infty})$
in Hochschild cohomology with the hyperplane $\frac{(\bar\tau-\tau)}{2} \vp_2 + b (\Ker B)$,
where $\vp_2$ is the generator of $HC^2(A_{\theta}^{\infty})$ given by 
\begin{equation}\label{genhc2}
   \vp_2(a,b,c) \, =\, \vp_0\left(a\, (\d_1 (b) \, \d_2(c)-\d_2 (b) \, \d_1(c))\right) , \quad
   a, b, c \in A_{\theta}^{\infty} \, .
\end{equation}
There is a canonical procedure (see \cite[\S VI. 3, Prop. 11]{Co-book}) for quantizing the
positive Hochschild cocycle $\phi$ thus obtained.
As analogue of the space of $(1,0)$-forms on the classical 2-torus
one takes the unitary bimodule
$\cH^{(1,0)}$ over $A_{\theta}^{\infty}$ given by the Hilbert space completion of the 
universal derivation bimodule $\Om^1(A_{\theta}^{\infty})$ 
of finite sums $\sum \, a \, d(b)$, $a,b \in A_{\theta}^{\infty}$, with respect to the inner product
\begin{equation}\label{onezero}
\langle a\, d(b) , a'\, d(b') \rangle = \vp_0 ((a')^* \, a \, \d (b) \, \d (b')^*) \, , \quad a,a',b,b' \in A_{\theta}^{\infty} \, .
\end{equation}
\begin{lem} \label{isoform}
The map $\psi:\cH^{(1,0)}\to \cH_0$,
\begin{equation}\label{isomap}
\cH^{(1,0)}\ni   \sum a d(b)  \mapsto  \sum a \d(b) \in \cH_0
\end{equation}
is a unitary $A_{\theta}^{\infty}$-bimodule isomorphism of $\cH^{(1,0)}$ with $\cH_0$.
\end{lem}
\proof
By definition of the inner product on $\cH^{(1,0)}$ the operator is unitary and the derivation property of $\d$ shows that it is an $A_{\theta}^{\infty}$-bimodule map. It remains to check that it is surjective. One has
$\d(U)=(\delta_1 + \bar{\tau} \delta_2)(U)=U$ and thus $\psi(aU^{-1}\partial U)=a$ which gives the required surjectivity. \endproof

When viewed as an unbounded operator from $\cH_0$ to $\cH^{(1,0)}$, the operator
$\d$ will be called $\partial$.

\subsection{Conformal changes of metric} \hfill\medskip

In order to implement conformal changes of metric,
we consider the family of positive linear functionals parameterized by
self-adjoint elements $h=h^* \in A_{\theta}^{\infty}$,  $\varphi = \varphi_h$, defined by
\begin{equation} \label{confstate}
\varphi (a) = \vp_0 (ae^{-h}) \, , \quad a \in A_{\theta} \, .
\end{equation}
\begin{defn}
{\em We shall call a positive linear functional $\vp$ on $A_{\theta}$
as in \eqref{confstate} a} conformal weight with
Weyl factor $e^{-h}$ and dilaton $h$. {\em The normalized functional }
\begin{equation} \label{confstate1}
\varphi_n (a) = \frac{\vp_0 (ae^{-h})}{\vp_0 (e^{-h})} \, , \quad a \in A_{\theta} \, .
\end{equation}
{\em is called the} associated conformal state.
\end{defn}

Each conformal weight $\vp$ determines an inner product $ \langle \ , \ \rangle_\varphi$ on $A_{\theta}$,
namely
\begin{equation}\label{confstate2}
\langle a,b\rangle_\varphi = \varphi (b^* a) \, , \quad a,b \in A_{\theta} \, .
\end{equation}
We let $\cH_\varphi$ denote the Hilbert space completion of $A_{\theta}$ for the inner product
$ \langle \ , \ \rangle_\varphi$. It is a unitary left module on $A_{\theta}$ by construction.
Note that, whereas for $\vp_0$ we have the trace relation
\begin{equation*}
\vp_0 (b^* a) = \vp_0 (ab^*) \, , \quad a,b \in A_{\theta} \, ,
\end{equation*}
the functional $\varphi$ satisfies instead
\begin{equation}\label{confstate3}
\varphi (ab) = \varphi (b e^{-h} ae^{h}) = \varphi (b \sigma_i \, (a)) \, , \quad a \in A_{\theta} \, ,
\end{equation}
which is  the KMS condition at $\beta=1$ for the 1-parameter group
 $\sigma_t$, $t \in {\mathbb R}$, of inner automorphisms
\begin{equation*}
\sigma_t (x) = e^{ith} x e^{-ith}
\end{equation*}
Equivalently, $\sigma_t=\modu^{-it}$  where the modular operator $\modu$, given by
\begin{equation*}
\modu (x) = e^{-h} x e^h \, , \qquad x \in A_{\theta}
\end{equation*}
is positive and fulfills
\begin{equation}\label{modoper}
    \langle \modu^{1/2}x, \modu^{1/2}x  \rangle_\varphi=\langle x^*,x^*   \rangle_\varphi
    \qqq \, x\in A_{\theta}.
\end{equation}
The infinitesimal generator of the
$1$-parameter group $\sigma_t$ is the inner derivation $-\lmod$,
\begin{equation*}
-\lmod(x)=-\log \modu(x) = [h,x] \, , \quad x \in A_{\theta}^{\infty} \, .
\end{equation*}
To correct the lack of
unitarity of the action of $A_{\theta}$ on $\cH_\varphi$ by right multiplication, one replaces it by the right action
 \begin{equation}\label{tomita}
 a \in A_{\theta} \mapsto a^{\rm op} := J_\varphi a^* J_\varphi \in \cL (\cH_\vp),
  \end{equation}
 where  $J_\varphi$ is the Tomita antilinear unitary of the GNS representation
 associated to $\vp$; explicitly, with $k=e^{h/2}$,
\begin{equation}\label{tomitaJ}
 J_\vp (a) \, = \modu^{1/2}( a^\ast)=\, k^{-1} a^\ast k  \qqq \, a \in A_{\theta} .
\end{equation}
One thus gets
\begin{equation}\label{tomita1}
    a^{\rm op}\xi =\xi k^{-1} a k \qqq a, \xi \in A_{\theta}^{\infty}.
\end{equation}
The obtained unitary $A_{\theta}^{\infty}$-bimodule is isomorphic to $\cH_0$,
\begin{lem} \label{isophi}
The right multiplication by $k$,
\begin{equation*}
    R_k a=ak  \qqq \, a\in A_{\theta}
\end{equation*}
 extends to an isometry $W :\cH_0 \rightarrow \cH_\varphi$
and gives a unitary $A_{\theta}^{\infty}$-bimodule isomorphism of  $\cH_0$ with $\cH_\varphi$.
\end{lem}
\proof One has
for any
 $a,b \in A_{\theta}$,
\begin{align*}
\langle  R_k(a), R_k(b)\rangle_\vp \, =\, \vp_0((bk)^*(ak)k^{-2}) \, = \,
\vp_0(b^*a)\, = \, \langle a, b\rangle.
\end{align*}
This shows that $W$ is an isometry. By construction it intertwines the left module structures. Moreover one has, using \eqref{tomita1},
$$
W(\xi a)=\xi a k=\xi k k^{-1} a k=W(\xi)k^{-1} a k= a^{\rm op}W(\xi)\qqq a, \xi \in A_{\theta}^{\infty}.
$$
This shows that $W$ intertwines the right module structures.\endproof

\subsection{Modular spectral triples on $\T^2_\theta$} \label{spectrip}
 \hfill\medskip

With the complex structure associated to $\tau \in \C$, $\Im (\tau) > 0$ fixed, the operator associated to the flat metric in the
corresponding conformal class on $\T^2_\theta$ is given by
\begin{equation}\label{eq:spectrip0}
 D= \left(
  \begin{array}{cc}
    0 & \partial^*\\
    \partial & 0\\
  \end{array}
\right)    \quad  \text{acting on } \quad \tilde{\cH} = \cH_0 \oplus \cH^{(1,0)} \, .
\end{equation}
In other words, this is the natural $\T^2_\theta$ version
of the $(\partial + \partial^*)$-operator, which
is isospectral to the usual Spin$_c$ Dirac operator
on the ordinary torus $\T^2$. The left and right actions for the unitary $A_\theta$-bimodule structure of $\tilde{\cH}$ both
give spectral triples. One can take the transpose in the sense of Definition \ref{transposetrip} of the  spectral triple $(A_\theta^{\rm op},\tilde{\cH},D)$ given by the right action of $A_\theta$. This transposed triple is isomorphic to the spectral triple given by the left action of $A_\theta$ in $\bar\cH=\cH_0 \oplus \cH^{(0,1)}$ and the operator
\begin{equation}\label{eq:spectrip0bis}
\bar D= \left(
  \begin{array}{cc}
    0 & \bar\partial^*\\
    \bar\partial & 0\\
  \end{array}
\right)    \quad  \text{acting on }  \cH_0 \oplus \cH^{(0,1)} \, .
\end{equation}
If one disregards the grading $\gamma$ the spectral triples $(A_\theta,\tilde{\cH},D)$
and $(A_\theta,\bar{\cH},\bar D)$ are equivalent but this does not hold as graded spectral triples and in fact the equivalence reverses the grading. One can see this distinction even in the commutative case by looking at the equation
$$
a[D,b]E=0\, , \ a,b\in \cA, \ \ E=\frac{1+\gamma}{2},
$$
which is fulfilled when $b$ is antiholomorphic.

We now
perform a non-trivial conformal change of metric on $\T^2_\theta$.
Let the conformal weight $\vp$ be as above.  The varying structure comes from the operator $\partial_\vp$
  which is given by  $\partial$ on $A_{\theta}^{\infty}$ but is  viewed as an unbounded operator from $\cH_\varphi$ to $\cH^{(1,0)}$,
  \begin{equation}\label{parti}
    \partial_\vp:A_{\theta}^{\infty}\subset\cH_\varphi\to \cH^{(1,0)}, \ \ \partial_\vp(a)=\partial(a)\qqq a\in A_{\theta}^{\infty}.
  \end{equation}
  In order to form the corresponding spectral triple we consider the operator
\begin{equation}\label{eq:spectrip}
 D_\vp= \left(
  \begin{array}{cc}
    0 & \partial_\vp^*\\
    \partial_\vp & 0\\
  \end{array}
\right)    \quad  \text{acting on } \quad \tilde{\cH}_\vp = \cH_\vp \oplus \cH^{(1,0)} \, ,
\end{equation}
 where we view  $\tilde{\cH}_\vp = \cH_\vp \oplus \cH^{(1,0)}$ both as a left module and a right module over  $A_{\theta}^{\infty}$.
Lemmas \eqref{isoform} and \eqref{isophi} show that an $A_{\theta}^{\infty}$-bimodule $\tilde{\cH}_\vp$ is isomorphic to $\cH=\cH_0\oplus \cH_0$ by the  unitary map
\begin{equation}\label{unitar}
    \tilde W(\xi,\eta)=(W(\xi), \psi^{-1}\eta)\in \cH_\vp \oplus \cH^{(1,0)}\qqq \xi,\eta \in \cH_0.
\end{equation}
Let $J$   denote
 the Tomita anti-unitary operator on $\cH_0$ extending the star involution
  $a \mapsto a^*$, $a\in A_{\theta}$. We let
   \begin{equation}\label{jtilde}
    \tilde J=
    \left(
      \begin{array}{cc}
        J & 0 \\
        0 & -J \\
      \end{array}
    \right)
   \end{equation}
    the direct sum of two copies of $\pm J$ acting in $\cH_0\oplus \cH_0$.
\begin{lem}\label{unitlem}
Let $k=e^{h/2}$, where $h=h^\ast \in A_{\theta}^\infty$ is the dilaton
of the conformal weight $\vp$.  We let $R_k$ denote the right multiplication by $k$ in $\cH_0$.
 \begin{itemize}
\item[(i)] The operator $\tilde W^* D_\vp \tilde W$ is equal to the self-adjoint unbounded operator
\begin{equation}\label{selfad}
    \tilde W^* D_\vp \tilde W= \left(
  \begin{array}{cc}
    0 & R_k \delta^*\\
    \delta R_k & 0\\
  \end{array}
\right) \,, \ \ \delta=\delta_1+\bar \tau \delta_2\,, \ \delta^*=\delta_1+ \tau \delta_2.
\end{equation}
\item[(ii)] The operator $\tilde J\tilde W^* D_\vp \tilde W \tilde J$ is equal to the self-adjoint unbounded operator
\begin{equation}\label{selfadbis}
   \tilde J\tilde W^* D_\vp \tilde W \tilde J= \left(
  \begin{array}{cc}
    0 & k \delta\\
    \delta^* k & 0\\
  \end{array}
\right)
\end{equation}
\end{itemize}
\end{lem}

\proof Let $\xi \in A_{\theta}^\infty\subset \cH_0$. One has $W(\xi)=\xi k=R_k\xi\in \cH_\vp$ and $\partial_\vp W(\xi)=\partial\circ R_k \xi$. Thus
$$
\psi(\partial_\vp W(\xi))=(\delta \circ R_k) \xi
$$
which gives the first statement. The second statement follows from the compatibility \eqref{invol} of the star operation with the derivations $\delta_j$. \endproof

\begin{cor}\label{speclem}
Let $k=e^{h/2}$, with $h=h^\ast \in A_{\theta}^\infty$ the dilaton
of the conformal weight $\vp$.
 \begin{itemize}
\item[(i)] The left action of $A_{\theta}$ on $\tilde{\cH}_\vp$ together with the operator
$D_\vp$ yield a graded spectral triple $(A_{\theta},\tilde{\cH}_\vp, D_\vp)$.

\item[(ii)]  The right action $a\mapsto a^{\rm op}$ of $A_{\theta}$ on $\tilde{\cH}_\vp$
together with the operator $D_\vp$ yield a graded twisted spectral triple
$(A^{\rm op}_{\theta},\tilde{\cH}_\vp ,D_\vp)$, with bounded twisted commutators
\begin{equation}\label{rightbounded}
   D_\vp \,a^{\rm op} - (k^{-1}ak)^{\rm op} D_\vp  \in \cL (\tilde{\cH}_\vp) \, \qqq \,
   a\in A^\infty_{\theta} .
\end{equation}
\item[(iii)] The transposed of the modular spectral triple $(A^{\rm op}_{\theta},\tilde{\cH}_\vp ,D_\vp)$ is isomorphic to the perturbed spectral triple
    \begin{equation}\label{perttrip}
        (A_{\theta},\cH, \dvp), \ \dvp=\left(
  \begin{array}{cc}
    0 & k \delta\\
    \delta^* k & 0\\
  \end{array}
\right)\sim (\bar D)_{(k,\gamma)}\,.
    \end{equation}
\end{itemize}
\end{cor}

\proof (i) In order to show that $[D_\vp, \, a]$ is bounded, it suffices
to check that $[\partial_\varphi,a]$ is bounded. In turn, the latter easily
follows from the derivation property of $\partial_\varphi$ and the equivalence
of the norms $\Vert.\Vert_\varphi$ and $\Vert.\Vert_0$.
\medskip

(ii)   This follows from Lemma \ref{tiwst1} and the third statement which we now prove.

(iii) This follows from the second statement of Lemma \ref{unitlem} using \eqref{reltom}.
\endproof

By Corollary \ref{speclem},  the transposed of the modular spectral triple $(A^{\rm op}_{\theta},\tilde{\cH}_\vp ,D_\vp)$ is simply given by the left action of $A_\theta$ on $\cH=\cH_0\oplus \cH_0$ and the operator
\begin{equation}\label{modspectrip}
\dvp=\left(
  \begin{array}{cc}
    0 & k \delta\\
    \delta^* k & 0\\
  \end{array}
\right)
\end{equation}

\begin{defn}\label{define}  The modular spectral triple of weight $\vp$ is
\begin{equation}\label{modspectrip1}
   (A_{\theta}^{\infty}, \cH, \dvp)
\end{equation}
where $\cH=\cH_0\oplus \cH_0$ as a left $A_{\theta}^{\infty}$-module and $\dvp$ is given by
\eqref{modspectrip}.
\end{defn}

\subsection{Laplacians on $\T^2_\theta$} \label{sectlap}
 \hfill\medskip

The spectral invariants of the modular spectral triple of weight $\vp$ are obtained by computing zeta functions and heat expansions, \ie traces of products of an element of $A_\theta$ (acting on the left) by a function of ${\dvp}^2$. Let $\lapa$ be
 the Dolbeault-Laplace operator
for the flat metric,
\begin{align} \label{lap0}
  \lapa\, = \, \delta\, \delta^*
  \, = \,\delta_1^2\, +\, 2 \Re(\tau)\delta_1\delta_2 \, + \, |\tau|^2 \delta_2^2 \,
\end{align}
acting on functions on $\T^2_\theta$.

\begin{lem}\label{laplem}
Let $k=e^{h/2}$, where $h=h^\ast \in A_{\theta}^\infty$ is the dilaton
of the conformal weight $\vp$.
 \begin{itemize}
\item[(i)] One has
\begin{equation}\label{dprimesq}
    \dvp^2=\left(
  \begin{array}{cc}
    k \lapa k & 0\\
    0 & \pertc\\
  \end{array}
\right)\,,  \ \lapa =\delta\,\delta^*, \ \pertc = \delta^* k^2\delta
\end{equation}
\item[(ii)]  The Laplacian  on functions is anti-unitarily equivalent to $\perta=k\lapa k$.
\item[(ii)] The operator $\pertc = \delta^* k^2\delta$  is anti-unitarily equivalent to the Laplacian $\pertb$ on forms of type $(1,0)$.
\end{itemize}
\end{lem}
\proof This follows from Corollary \ref{speclem}.\endproof

\begin{lem} \label{lem:3zetas}
Let $\vp$ be a conformal weight with dilaton $h = h^\ast \in A_\theta^\infty$.
 The zeta function of the Laplacian on functions is equal to the zeta functions of the   operators $\perta$, $\pertb$ and $\pertc$:
\begin{align} \label{eq:3zetas}
\zeta_{\perta}(z) \, = \, \zeta_{\pertb}(z)
 \, = \,\zeta_{\pertc}(z)= \zeta_{k \lapa k} (z) .
\end{align}
\end{lem}

\proof The operators $\perta =k\lapa k$ and
$\pertc = \delta^* k^2\delta$ have the same spectrum outside
$0$, which proves the first equality in \eqref{eq:3zetas}. The others
follow from Lemma \ref{laplem}.
\endproof

\section{Conformal invariants} \label{S:conf-inv}

\subsection{Conformal index of a spectral triple} \label{subs:conf-ind}  \hfill\medskip

We digress a little to show that the notion of {\em conformal index}
for a manifold, introduced in~\cite{bo1},
 admits a natural extension to the framework of noncommutative geometry.
\medskip

Let $(\cA , \fH , D)$ be a $p$-summable spectral triple, which
has {\em discrete dimension spectrum} in the sense of~\cite{lif}.
 Fix $h = h^* \in \cA$, and let
 \begin{equation}\label{conf-def}
 D_{sh} \, =\,e^{ \frac{s h}{2} }\,D\,e^{\frac{s h}{2} }       , \qquad  s \in \R .
\end{equation}
Then
  \begin{align*}
 \frac{d}{ds}D_{sh}\,= \, \frac{1}{2} (h D_s  + D_s  h)  ,
 \end{align*}
 hence
   \begin{align*}
 \frac{d}{ds}D_{sh}^2 &\,= \,  \frac{1}{2} \left(h D^2_{sh}  + 2 D_s  h D_s   +   D^2_{sh} h\right)
\end{align*}
 Duhamel's formula for the family $\triangle_s = t D_{sh} ^2$,
  \begin{equation} \label{duhamel}
\frac{de^{- \triangle_s} }{d s} \, = \, - \int_0^1 e^{-u \triangle_s} \,
\frac{d \triangle_s}{d s} \, e^{- (1-u) \triangle_s} \, du \, ,
\end{equation}
allows  to write
\begin{eqnarray} \notag
\frac{d}{d s} \Tr \left(e^{-t D^2_{sh} } \right) &=- \frac{t}{2}\, \Tr \left((h D^2_{sh}  +
2 D_{sh}  h D_{sh}   +   D^2_{sh} h) e^{-t D^2_{sh} }\right) \\ \label{useduhamel}
&=- 2 t  \, \Tr \left(h \, D^2_{sh} \,e^{-t D^2_{sh} } \right)  .
 \end{eqnarray}
Noting that
   \begin{equation*}
- \, \Tr \left(h \, D^2_{sh} \,e^{-t D^2_{sh} } \right)
\, = \,  \,\frac{d}{dt}  \Tr \left(h \, e^{-t D^2_{sh} } \right) ,
\end{equation*}
one obtains the identity
 \begin{equation} \label{exchange}
\frac{d}{d s} \Tr \left(e^{-t D^2_{sh} } \right) \, = \, 2 t  \,\frac{d}{dt}  \Tr \left(h \, e^{-t D^2_{sh} } \right) .
\end{equation}

At this point we make an additional assumption, which stipulates
 \textit{the existence of small time asymptotic expansions of the form
 \begin{equation} \label{asymp1}
  \Tr \left(e^{-t D^2_{sh}} \right)\, \sim_{t \searrow 0} \,
   \sum_{j=0}^\infty {\rm a}_j(D^2_{sh})\, t^{\frac{j - p}{2}}  ,
   \end{equation}
  and more generally, for any $f \in \cA$,
 \begin{equation} \label{asymp+}
  \Tr \left(f\, e^{-t D^2_{sh}} \right)\, \sim_{t \searrow 0} \,
   \sum_{j=0}^\infty {\rm a}_{j}(f, D^2_{sh})\, t^{\frac{j - p}{2}}  ,
   \end{equation}
which moreover can be differentiated term-by-term with respect to
$s \in [-1, 1]$}.

\begin{thm} \label{thmconfind}
Under the above assumptions
the value of the zeta function at the origin $\zeta_{|D|} (0)$ is invariant under conformal
deformations \eqref{conf-def} of the spectral triple $(\cA , \fH , D)$.
\end{thm}

\proof
Denote by $|D_{sh}|^{-1}$ the inverse of $|D_{sh}|(1-P_{sh})$ restricted to $\Ker (D_{sh})^\perp$,
where $P_{sh}$ stands for the orthogonal projection onto $\Ker (D_{sh})$,
and consider the zeta function
\begin{align*}
\zeta_{|D_{sh} |} (z) = \Tr ( |D_{sh}|^{-z}) , \quad  \Re z >  p ,
\end{align*}
which is
related to the theta function by the Mellin transform
\begin{align} \label{Mell1}
 \zeta_{|D_{sh}|} (2z) \, = \, \frac{1}{\Gamma (z)} \int_0^\infty t^{z-1} \,
  \left(\Tr (e^{-t D^2_{sh}}) -\dim \Ker D_{sh} \right) \, dt\, .
\end{align}
The asymptotic expansion \eqref{asymp1} ensures that $\zeta_{|D_{sh}|} (z)$ has
meromorphic continuation to $\mathbb{C}$, with only simple poles. Furthermore,
 because of the pole of $\Gamma (z)$ at $z=0$,  $\zeta_{|D_{sh}|} (z)$ is holomorphic at $0$,
and its value at $0$ is
\begin{align} \label{zeta0}
 \zeta_{|D_{sh}|} (0) \, = \,  {\rm a}_p (D^2_{sh}) -\dim \Ker D_{sh} \, = \,
 {\rm a}_p (D^2_{sh}) -\dim \Ker D .
\end{align}
Differentiating term-by-term the asymptotic expansion  \eqref{asymp1}
and applying \eqref{exchange} yields the identities
 \begin{equation} \label{comp1}
\frac{d}{d s}{\rm a}_{j}(D^2_{sh})  \, = \, (j-p) \, {\rm a}_{j}(h, D^2_{sh}) \, ,
\quad j \in \mathbb{Z}^+ .
\end{equation}

In particular,
 \begin{equation*}
\frac{d}{d s} {\rm a}_p (D^2_{sh})  \, = \, 0 \, ,
\end{equation*}
hence
 \begin{equation} \label{c-index}
 \zeta_{|D_{sh}|} (0) \, = \,  {\rm a}_p (D^2) -\dim \Ker D  \, = \,  \zeta_{|D|} (0) .
\end{equation}
\endproof

\medskip

An instance where the above hypotheses are satisfied, and hence the result applies,
is that of the dilaton field rescaling of the mass in the spectral action formalism
for the standard model~\cite{cc3}.

\subsection{Conformal index for $\T^2_\theta$} \label{s:conf-ind-tor}  \hfill\medskip

More to the point, the pseudodifferential calculus for $C^\ast$-dynamical
systems~\cite{C}, and especially the elliptic theory on noncommutative
tori~\cite[\S IV.6]{Co-book}, show that the condition \eqref{asymp+} is  fulfilled
in the case of $\T^2_\theta$. In particular, all the Laplacians in \S \ref{spectrip}
admit meromorphic zeta functions, which have simple poles and are regular at $0$.

\begin{thm} \label{thmconfindlap}
The value at the origin of the zeta function of the Laplacian on functions
 is a conformal invariant,  \ie
\begin{align} \label{confindtor}
\zeta_{\pert} (0) \, = \, \zeta_{\lap}(0) ,
\end{align}
for any conformal weight $\vp$  on $A_\theta$.
\end{thm}

\proof
In view of  Lemma \ref{laplem}, one can replace the Laplacian on functions  by $\perta=k \,\lapa \, k$.
Consider the family
 \begin{equation} \label{lapfam}
 \triangle_{sh} \, =\,e^{ \frac{s h}{2} }\,\lapa\,e^{\frac{s h}{2} }       , \qquad  s \in \R .
\end{equation}
Since
  \begin{align} \label{lapfam'}
 \frac{d}{ds}\triangle_{sh} \,= \, \frac{1}{2} \left(h \triangle_{sh}  + \triangle_{sh}  h\right) ,
 \end{align}
by using Duhamel's formula as in \eqref{useduhamel}, one sees
\begin{eqnarray*}
\frac{d}{d s} \Tr \left(e^{-t \triangle_{sh} } \right) =\, -  t  \, \Tr \left(h \, \triangle_{sh} \,
e^{-t \triangle_{sh} } \right)
\, = \, t  \,\frac{d}{dt}  \Tr \left(h \, e^{-t \triangle_{sh} } \right) .
 \end{eqnarray*}
 The variation formulas for
 the coefficients of the heat operator asymptotic expansion yield in this case
 the identities
   \begin{equation} \label{compL}
\frac{d}{d s} {\rm a}_{j}( \triangle_{sh} )  \, = \,
\frac{1}{2} (j-2) \, {\rm a}_{j}(h, \triangle_{sh} ) \, , \quad j \in \mathbb{Z}^+ .
\end{equation}
In particular, ${\rm a}_{2}( \triangle_{sh} ) = {\rm a}_{2}( \lapa)$,
and the proof is achieved in the same way
as that of Theorem \ref{thmconfind}.
\endproof

\begin{rem} {\rm This gives a non-computational proof to the Gauss-Bonnet theorem for the
 noncommutative $2$-torus (\cf~\cite{cc}, \cite{FK}).}
 \end{rem}

\section{Zeta functions and local invariants}\label{main}

We now focus on the zeta function of the modular spectral triple of weight $\vp$, as in Definition \ref{define}, \ie  $
(A_{\theta}^{\infty}, \cH, \dvp)$,
in order to compute its local invariants. In order to state our first main result, \ie Theorem \ref{thmmain}, we shall first introduce several functions which play a key role in the statement of the basic formula.

\begin{figure}
\begin{center}
\includegraphics[scale=0.7]{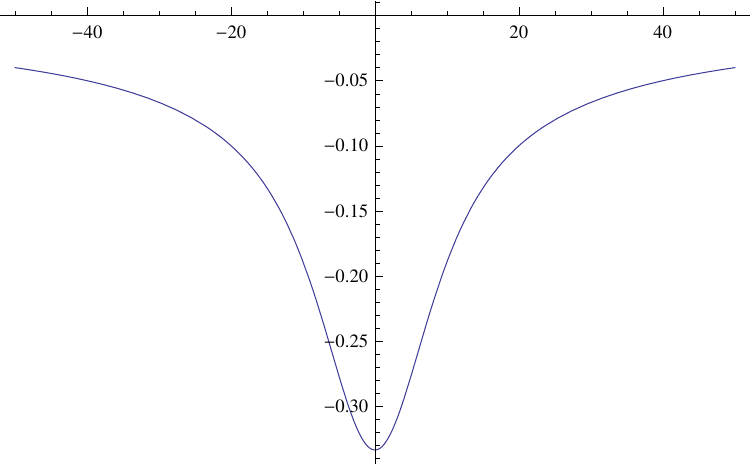}
\caption{Graph of the  function $K$, even and negative with $K(0)=-\frac 13$.}\label{Kfunction}
\end{center}
\end{figure}

\subsection{Curvature functions}\hfill\medskip

In the formulation of Theorem \ref{thmmain} there is an overall factor of $-\frac{\pi}{\tau_2}$ where $\tau_2=\Im(\tau)$ is the imaginary part of $\tau$
but the other functions involved are independent of $\tau$ and we list them below and analyze their elementary properties.

\subsubsection{Functions of one variable}

One always gets an expression of the form $K(\nabla)$ applied to
$$
\lapa (\ell) = \delta_1^2(\ell)+2 \Re(\tau)\delta_1\delta_2(\ell)+|\tau|^2 \delta_2^2(\ell) ,
 \quad \ell = \log k .
$$

For the first half of Laplacian, the function  is
$$
K_0(s)=\frac{2 e^{s/2} \left(2+e^s (-2+s)+s\right)}{\left(-1+e^s\right)^2 s}
=\frac{-2+s\,{\rm coth}(s/2)}{s\,\sh(s/2)}\, .
$$

For the full Laplacian, it is
$$
K(u)= \, \frac{\frac 12-\frac{\sh (u/2)}{u}}{\sh^2 (u/4)}\, .
$$
For the graded case
$$
K_\gamma(u) = \, \frac{\frac 12+\frac{\sh(u/2)}{u}}{\ch^2(u/4)}\, .
$$

\begin{figure}
\begin{center}
\includegraphics[scale=0.7]{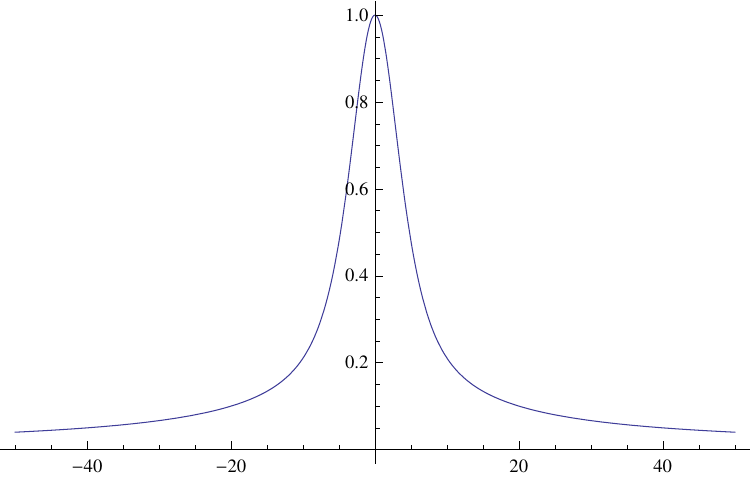}
\caption{Graph of the  function $K_\gamma$, even and positive with $K_\gamma(0)=1$.}
\label{Kgfunction}
\end{center}
\end{figure}

\subsubsection{Functions of two variables}

One always gets an expression of the form $H(\nabla_1,\nabla_2)$ applied to
\begin{equation}\label{basicquadra}
\square_\Re (\ell) :=
(\delta_1(\ell))^2+ \Re(\tau)\left(\delta_1(\ell)\delta_2(\ell)+\delta_2(\ell)\delta_1(\ell)\right)+|\tau|^2 (\delta_2(\ell))^2 , \quad \ell = \log k .
\end{equation}

The functions $H$ of two variables  are:

For the first half of Laplacian
%\begin{tiny}
\begin{align}\label{basicquadra1}
\begin{split}
&H_0(s,t)= \\
&\frac{t (s+t) \ch(s)-s (s+t) \ch(t)+(s-t) (s+t+\sinh(s)+\sinh(t)-\sinh(s+t))}{s t (s+t) \sinh\left(\frac{s}{2}\right) \sinh\left(\frac{t}{2}\right) \sinh\left(\frac{s+t}{2}\right)^2}
\end{split}
\end{align}
%\end{tiny}
For the full Laplacian, the formula of Theorem \ref{thmmain} involves $H_0+H_1$ where
\begin{equation}\label{h1def}
H_1(s,t)=\ch\left(\frac{s+t}{2}\right)H_0(s,t)
\end{equation}
For the graded case, it involves $H_0-H_1$.

\subsubsection{Skew term}

The additional skew term
% involving the commutator of the $\delta_j(\log k)$
is of the form $S(\nabla_1,\nabla_2)$ applied to
\begin{equation}\label{skew}
\square_\Im (\ell) :=
i\,  \Im (\tau)\left(\delta_1(\ell)\delta_2(\ell)-\delta_2(\ell)\delta_1(\ell)\right)  , \quad \ell = \log k .
\end{equation}
It only appears in the second half of Laplacian. The function $S$ is
\begin{equation}\label{skew1}
S(s,t)=\frac{(s+t-t\, \ch(s)-s\, \ch(t)-\sinh(s)-\sinh(t)+\sinh(s+t))}{s\, t\left(\sinh\left(\frac{s}{2}\right) \sinh\left(\frac{t}{2}\right) \sinh\left(\frac{s+t}{2}\right)\right)}
\end{equation}
which is a symmetric function of $s$ and $t$.

\subsubsection{Elementary properties}

We now list the elementary properties of the modular curvature functions of two variables.
(\cf Figures \ref{h0function}, \ref{h1function}, \ref{Sfunction})

\begin{lem} \label{scacurv}
The functions $H_0(s,t)$,  $H_1(s,t)=\ch\left(\frac{s+t}{2}\right)H_0(s,t)$ and $S(s,t)$ fulfill the following properties
\begin{enumerate}
  \item They belong to  $C^\infty_0(\R^2)$.
  \item $H_j(t,s)=-H_j(s,t)$, $S(t,s)=S(s,t)$ and $S(s,t)\geq 0$.
  \item $H_j(-s,-t)=-H_j(s,t)$, $S(-s,-t)=S(s,t)$.
\end{enumerate}
\end{lem}

\proof The smoothness of $H_0$ is clear outside the three lines $L_1:s+t=0$, $L_2:s=0$ and $L_3:t=0$.
 Let $n(s,t)$ be the numerator of the fraction defining $H_0$. Near the first line one gets
 the expansion
 $$
 n(s,t)=\frac{1}{6} (-2 s-s \ch(s)+3 \sinh(s)) (s+t)^3+\frac{1}{12} (2-2 \ch(s)+s \sinh(s)) (s+t)^4$$ $$+\frac{1}{120} (-2 s-3 s \ch(s)+5 \sinh(s)) (s+t)^5+O(s+t)^6
 $$
 which shows that as long as $(s,t)\neq (0,0)$ the function $H_0$ is smooth at $(s,t)\in L_1$.
 The value of $H_0$ on $L_1$ is given by
 \begin{equation}\label{valuel1}
 H_0(s,-s)=   -\frac{4 \left(3-3 e^{2 s}+s+4 e^s s+e^{2 s} s\right)}{3 \left(\left(-1+e^s\right)^2 s^2\right)}
 \end{equation}
\begin{figure}
\begin{center}
\includegraphics[scale=1.0]{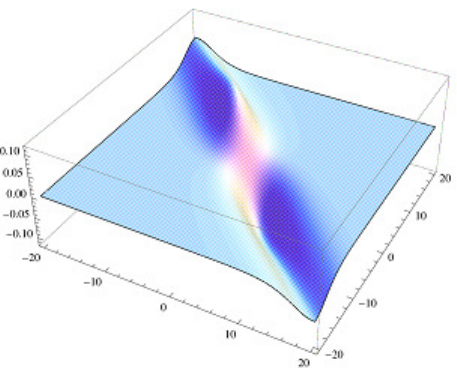}
\caption{Graph of the  function $H_0$.}\label{h0function}
\end{center}
\end{figure}

 Near the second line $L_2$ one gets the expansion
 $$
 n(s,t)=\frac{1}{2} \left(4+t^2-4 \ch(t)+t \sinh(t)\right) s^2+\frac{1}{6} (2 t+t \ch(t)-3 \sinh(t)) s^3+O(s)^4
 $$
 which gives the smoothness for $(s,t)$ on the second line (and the third similarly). One needs to look carefully at what happens at the crossing point $(s,t)=(0,0)$.
 One finds that the Taylor expansion of $H_0$ at the point $(0,0)$ is of the form
$$
H_0(s,t)=-\frac{s}{45}+\frac{t}{45}+\frac{s^3}{504}+\frac{s^2 t}{840}-\frac{s t^2}{840}-\frac{t^3}{504}-\frac{s^3 t^2}{6720}+\frac{s^2 t^3}{6720}$$ $$-\frac{47 s^4 t}{201600}+\frac{47 s t^4}{201600}-\frac{67 s^5}{604800}+\frac{67 t^5}{604800}+\ldots
$$

\medskip

Similarly we let $m(s,t)$ be the numerator of the fraction defining $S(s,t)$
 $$
 m(s,t)=(s+t-t\, \ch(s)-s\, \ch(t)-\sinh(s)-\sinh(t)+\sinh(s+t))
 $$
 and get the expansion near the line $L_1$ in the form
$$
m(s,t)=(2-2 \ch(s)+s \sinh(s)) (s+t)+\frac{1}{2} (-s \ch(s)+\sinh(s)) (s+t)^2$$ $$+\frac{1}{6} (1-\ch(s)+s \sinh(s)) (s+t)^3+O(s+t)^4
$$
Near the second line $L_2$ one gets the expansion
$$
m(s,t)=\frac{1}{2} (-t+\sinh(t)) s^2+\frac{1}{6} (-1+\ch(t)) s^3+O[s]^4
$$
The denominator of $S$ is
$$
s\, t\,\sinh\left(\frac{s}{2}\right) \sinh\left(\frac{t}{2}\right) \sinh\left(\frac{s+t}{2}\right)
$$
and this gives the smoothness outside the origin. At $(0,0)$ one has the Taylor expansion
$$
S(s,t)=\frac{2}{3}-\frac{s^2}{45}-\frac{s t}{30}-\frac{t^2}{45}+\frac{s^4}{1260}+\frac{s^3 t}{504}+\frac{s^2 t^2}{378}+\frac{s t^3}{504}+\frac{t^4}{1260}+\ldots
$$

\medskip

\begin{figure}
\begin{center}
\includegraphics[scale=1.1]{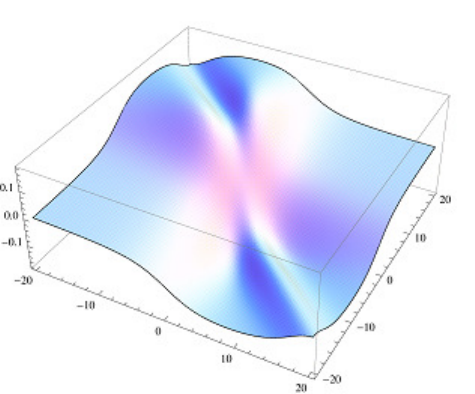}
\caption{Graph of the  function $H_1$.}\label{h1function}
\end{center}
\end{figure}

We now look at the behavior at $\infty$. It is enough to show that the function
 $H_1(s,t)=\ch\left(\frac{s+t}{2}\right)H_0(s,t)$ tends to $0$ at $\infty$.
 We write $H_1(s,t)$ as the fraction
 \begin{equation}\label{numdemh1}
   \frac{t (s+t) \ch(s)-s (s+t) \ch(t)+(s-t) (s+t+\sinh(s)+\sinh(t)-\sinh(s+t))}{s t (s+t) \sinh\left(\frac{s}{2}\right) \sinh\left(\frac{t}{2}\right) {\rm tanh}\left(\frac{s+t}{2}\right)\sinh\left(\frac{s+t}{2}\right)}
 \end{equation}
  First note the equality
\begin{equation}\label{maxsum}
    \sup\{|s|,|t|,|s+t|\}=\frac 12\left( |s|+|t|+|s+t|    \right)\qqq s,t\in \R.
\end{equation}
which shows that away from the lines $L_j$ the numerator and denominator have the same exponential increase. Let $||(s,t)||_1=|s|+|t|$. The maximum of $|H_1(s,t)|$ on the sphere $S_a=\{(s,t)\mid ||(s,t)||_1=a\}$ is reached on the interval
$$
I_a=\{(s,-a+s)\mid s\in [\frac a2,a]\}
$$

\begin{figure}
\begin{center}
\includegraphics[scale=1.1]{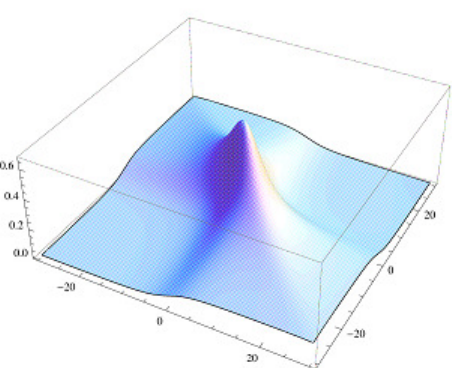}
\caption{Graph of the  function $S$.}\label{Sfunction}
\end{center}
\end{figure}

Away from the boundary of this interval one can approximate the denominator of $H_1$  by the product of the leading exponentials which gives
$$
\frac 18 e^{\frac{s}{2}-\frac{t}{2}+\frac{s+t}{2}} s t (s+t)
$$
Using this approximation and neglecting terms which are suppressed by an exponential
one reduces the function $H_1(s,t)$ well inside the interval $I_a$ to the fraction
$$
r(s,t)=-\frac{4 ((-1+t) t+s (1+t))}{s t (s+t)}
$$
On $I_a$ the function $r(s,t)$ reaches its minimum at $(s,t)=(\frac a2+ v\sqrt{a} ,-\frac a2+ v\sqrt{a} )$ where $v\sim \frac{1}{\sqrt 2}$ fulfills
$$
1-2 v^2-\frac{12 v^2}{a}+\frac{8 v^3}{\sqrt{a}}-\frac{8 v^4}{a}=0
$$
One finds in this way that the maximum of $|H_1(s,t)|$ on the sphere $S_a=\{(s,t)\mid ||(s,t)||_1=a\}$ is of the order of $\frac 8a$. One needs to control the size of the function
$H_1(s,t)$ in the neighborhood of the zeros of the denominator.
The restriction of $H_1(s,t)$ to the anti-diagonal $t=-s$ is given by the odd function
$$
H_1(s,-s)=(-\frac{4 s}{3}-\frac{2}{3} s \ch(s)+2 \sinh(s))/s^2\sh(s/2)^2
$$
which is equivalent to $-\frac 43 \frac 1s$ when $s\to\pm \infty$.
The restriction of $H_1(s,t)$ to the axis $t=0$ is given by
$$
H_1(s,0)=-\frac{(4+s^2-4 \,\ch(s)+s \,\sh(s))\ch(s/2)}{s^2\sh(s/2)^3}
$$
which is equivalent to $- \frac 2s$ when $s\to\pm \infty$.

The restriction of $S(s,t)$ to the anti-diagonal $t=-s$ is given by
$$
\frac{4 \left(-2+s \,\text{cotanh}(\frac{s}{2})\right)}{s^2}
$$
which behaves like $\frac 4s$ for $s\to \pm \infty$ and this gives the behavior of the
maximum of $|S(s,t)|$ on the sphere $S_a=\{(s,t)\mid ||(s,t)||_1=a\}$. The minimum of $S(s,t)$ on the sphere $S_a$ is reached on the diagonal $s=t$ where the function reduces to
$$
S(x,x)=\frac{4}{x^2}-\frac{4}{x \sinh(x)}\geq 0
$$
The other properties can be checked in a straightforward manner.
\endproof

\subsection{Local curvature functionals} \hfill\medskip

 With the above notation, we are now ready to express in local terms the value at the
 origin of the zeta functions of the modular spectral triple of weight $\vp$
 $$
(A_{\theta}^{\infty}, \cH, \dvp).$$

For each $a\in A_{\theta}^{\infty}$ we consider the  zeta function
\begin{align} \label{zetaa}
\zeta_{\pert} (a, z) = \Tr (a\, \pert^{-z}\, (1-P_\vp)) , \quad  \Re z >  2 ,
\end{align}
where $P_\vp$ stands for the orthogonal projection onto $\Ker \pert$. It is
related to the theta function by the Mellin transform
\begin{align} \label{Mell2}
 \zeta_{\triangle_\vp} (a, z) \, = \, \frac{1}{\Gamma (z)} \int_0^\infty t^{z-1} \,
  \left(\Tr (a\, e^{-t \triangle_\vp}) -\Tr (P_\vp \, a \, P_\vp) \right) \, dt\, .
\end{align}
As in the untwisted case, \cf \eqref{zeta0}, its value at $0$ is related to the
constant term in the asymptotic expansion \eqref{asymp1}, via
\begin{align}  \label{zetaa0}
 \zeta_{\triangle_\vp}(a, 0) \, = \,  {\rm a}_2 (a, \triangle_\vp) \, - \, \Tr (P_\vp \, a \, P_\vp) .
\end{align}
The computation of the constant term $ {\rm a}_2 (a, \triangle_\vp)$ is quite
 formidable, as could be expected from the already laborious calculations
 performed in \cite{cc} and \cite{FK} in the untwisted case, \ie $a=1$.
  For the clarity of the exposition,
we postpone giving the technical details until \S \ref{tech}.

 On the other hand, the additional term is very easy to compute. Indeed,
 $\Ker \perta = \Ker (\delta k)$ is one dimensional and one has (with $\vp_n$ the associated state)
 \begin{equation}\label{projker0}
   \Tr (P_\vp \, a \, P_\vp)=\vp_0(ak^{-2})/\vp_0(k^{-2})
 = \frac{\vp (a)}{\vp(1)}=\vp_n(a) \,.
 \end{equation}
 One deals in a similar manner with the Laplacian $\pertc$ and one lets $P^{(0,1)}$ be the orthogonal projection on its one-dimensional kernel, $\Ker \pertc = \Ker( k \delta^*)$ which is independent of $k$,
 and
consists of the constant multiples of the unit $1 \in A_\theta$, so that
\begin{equation} \label{projker}
 \Tr (P^{(0,1)} \, a \, P^{(0,1)}) =  \vp_0 (a) .
\end{equation}

\begin{thm}\label{thmmain}
Let $\vp$ be a conformal weight with dilaton
  $h = h^\ast \in A_\theta^\infty$ and let $k=e^{h/2}$.
The value at the origin of the zeta function associated to the modular
spectral triple of weight $\vp$
is given for any $a\in A_{\theta}^{\infty}$ by the expression
\begin{eqnarray} \notag
    \Tr(a|\dvp|^{-z})\big|_{z=0}&=&-\frac{\pi}{\tau_2}\vp_0(a\left(K(\lmod)(\lapa(\log k))+H(\nabla_1,\nabla_2)(\square_\Re (\log k)\right.) \\ \label{zetavalue1}
   &+&\left. S(\nabla_1,\nabla_2)
(\square_\Im (\log k))\right) -\vp_n(a)  \, - \, \vp_0 (a) \, ,
 \end{eqnarray}
 where $H=H_0+H_1$
 and for its graded version by
\begin{eqnarray} \notag
\Tr(\gamma a|\dvp|^{-z})\big|_{z=0}&=&-\frac{\pi}{\tau_2}\vp_0(a\left(K_\gamma(\lmod)(\lapa(\log k))+
   H_\gamma(\nabla_1,\nabla_2)(\square_\Re (\log k) \right.) \\ \label{zetavalue2}
   &-&\left. S(\nabla_1,\nabla_2)
(\square_\Im (\log k))\right) -\vp_n(a) \, + \, \vp_0 (a)   .
 \end{eqnarray}
 where $H_\gamma=H_0-H_1$.
\end{thm}
\medskip

In order to check the normalization constants we compare this result with the classical formula
for the value at $0$ of the zeta function of the Laplacian on a closed surface $\Sigma$
\begin{equation}\label{lapclass0}
\zeta (0)+{\rm Card} \{ j \mid \lambda_j = 0 \} = \frac{1}{12 \pi}
\int_\ries \, R \sqrt{g}d^2x = \frac 16 \,\chi (\ries)  \, ,
\end{equation}
where $R$ is the scalar curvature (normalized as being $1$ for the unit two sphere) and $\chi (\ries) $ the
 Euler-Poincar\'e characteristic. One double checks this formula for the unit two sphere, whose Laplacian spectrum is the set $\{n^2+n\mid n\in \Z_+\}$ where the eigenvalue $n^2+n$ has multiplicity $2n+1$. One has
 $$
 \sum_{\Z_+}(2n+1)e^{-t(n^2+n)}\sim \frac 1t +\frac 13 +O(t)
 $$
 whose constant term $\frac 13$ agrees with the right hand side $\frac 16 \,\chi (S^2)$ of \eqref{lapclass0}.
   In a local form one has
 \begin{equation}\label{lapclass}
   \Tr(a\Delta^{-z})\big|_{z=0}+\Tr(aP)=\frac{1}{12 \pi}
\int_\ries a\, R \sqrt{g}d^2x
 \end{equation}
 where $P$ is the orthogonal projection on the kernel of the Laplacian. To compare this formula with Theorem \ref{thmmain} we take the half sum of \eqref{zetavalue1} and \eqref{zetavalue2} in
 the commutative case. This reduces to
 \begin{equation}\label{comcase}
   \Tr(a\Delta^{-z})\big|_{z=0}+\Tr(aP)=
   -\frac{\pi}{\tau_2}\vp_0(a K_0(0)\lapa(\log k))
 \end{equation}
 One has $K_0(0)=\frac 13$ and $\vp_0$ is the state associated to the volume form of the flat two torus with integral spectrum. To check the overall normalization we take $\tau=i$ so that $\tau_2=1$. The torus has coordinates $x_j$ with period $2\pi$, Riemannian metric $ds^2=dx_1^2+dx_2^2$, and the spectrum of the Laplacian is the set $\{n^2+m^2\mid n,m\in \Z\}$. Thus
 \begin{equation}\label{phizer}
   \vp_0(a)=\frac{1}{(2\pi)^2}\int a dx_1dx_2\,, \ \ \lapa(f)=-(\partial_1^2+\partial_2^2)f\,.
 \end{equation}
 We now consider the Riemannian metric $g=k^{-2}(dx_1^2+dx_2^2)$. Its volume form is $\sqrt{g}
 d^2x= k^{-2}dx_1dx_2$. The scalar curvature is
 \begin{equation}\label{scalcurvconf}
    R= k^2 (\partial_1^2+\partial_2^2)\log k\,.
 \end{equation}
 To check this, take the stereographic coordinates on the unit two sphere, so that the metric becomes $g=k^{-2}(dx_1^2+dx_2^2)$ with $k=\frac 12(1+x_1^2+x_2^2)$. Then \eqref{scalcurvconf} gives $R=1$ as required. Thus \eqref{lapclass} gives for arbitrary $k$ the local form
 \begin{equation}\label{lapclass1}
   \Tr(a\Delta^{-z})\big|_{z=0}+\Tr(aP)=\frac{1}{12 \pi}
\int_{T^2} a\, (\partial_1^2+\partial_2^2)(\log k) d^2x
 \end{equation}
 and this agrees with the right hand side of \eqref{comcase} which is, using \eqref{phizer},
 $$
 -\frac{\pi}{\tau_2}\vp_0(a K_0(0)\lapa(\log k))=\pi \frac 13 \vp_0(a\, (\partial_1^2+\partial_2^2)(\log k))=
 \frac{1}{12 \pi}
\int_{T^2} a\, (\partial_1^2+\partial_2^2)(\log k) d^2x
 $$

The fact that the local curvature expressions, occurring in \eqref{zetavalue1} on the one hand and
those occurring in \eqref{zetavalue2} on the other hand, are sharply different stands in stark
contrast with the case of the ordinary torus. For the latter they reduce to
\begin{align*}
-K(0)\, \lapa(\log k) \, k^2 \, = \,
\frac 16\, \lapa(h)\,e^h,  \quad \text{resp.} \quad - K_\g(0) \, \lapa(\log k) \, k^2 \, = \,
-\frac 12 \, \lapa(h) \, e^h ,
\end{align*}
and thus are both constant multiples of the
Gaussian curvature of the conformal metric.

\section{Log-determinant functional and scalar curvature}

In this section we develop the analogue of the Osgood-Phillips-Sarnak
functional~\cite{OPS}, which is a scale invariant version of the log-determinant
of the Laplacian. We then compute its gradient, whose corresponding flow for Riemann
surfaces exactly reproduces Hamilton's Ricci flow~\cite{Ham}, and therefore yields
the appropriate analogue of the scalar curvature.

\subsection{Variation of the log-determinant} \hfill\medskip

The Ray-Singer zeta function regularization of the determinant of a Laplacian~\cite{RS1},
as well as the related notion of analytic torsion~\cite{RS2}, make perfect sense
in the case of noncommutative $2$-tori, due to the existence of the appropriate
pseudodifferential calculus~\cite{C}. Thus,
\begin{align*}
 \log \Det^\prime (\perta) \, =\,  - \zeta^\prime_{\perta} (0) , \quad
 \text{resp.} \quad  \log \Det^\prime (\pertc) \, =\,  - \zeta^\prime_{\pertc} (0) ,
\end{align*}
are well-defined. Because $\perta = k\delta \delta^* k$
 and $\pertc = \delta^* k^2 \delta$ have the same spectrum outside $0$,
 and so the corresponding zeta functions coincide, the two log-determinants are
 in fact equal.
 \medskip

In order to compute the above determinant, let us
 consider again the $1$-parameter family of Laplacians, \cf \eqref{lapfam},
 \begin{equation*}
 \triangle_{sh} := \, k^s\,\lapa\,k^s
 \, =\,e^{ \frac{s h}{2} }\,\lapa\,e^{\frac{s h}{2} }       , \qquad  s \in \R .
\end{equation*}
Differentiating the corresponding family of zeta functions
and taking into account \eqref{lapfam'}  one
obtains, for $\Re z >  p\,$,
\begin{align*}
\begin{split}
&\frac{d}{ds} \zeta_{\triangle_{sh}} (z)\,= \, \frac{1}{\Gamma (z)} \int_0^\infty t^{z}
 \,\frac{d}{dt}  \Tr \left(h e^{-t \triangle_{sh}} (1-P_{sh})\right) \, dt  \\
&= \frac{1}{\Gamma (z)} \, t^{z}  \, \Tr \left(h e^{-t \triangle_{sh}} (1-P_{sh})\right) \bigg|_0^\infty \, - \,
  \frac{z}{\Gamma (z)} \int_0^\infty t^{z-1} \,  \Tr\left(h e^{-t \triangle_{sh}} (1-P_{sh})\right)\, dt \\
 &=\, -   \frac{z}{\Gamma (z)} \int_0^\infty t^{z-1} \,  \Tr\left(h e^{-t \triangle_{sh}} (1-P_{sh})\right)\, dt
 \, = : \, - z \, \zeta_{\triangle_{sh}} (h, z).
  \end{split}
\end{align*}
By meromorphic continuation one obtains the identity
\begin{align} \label{dMellLap}
\frac{d}{ds} \zeta_{\triangle_{sh}} (z)\, = \,
- z \, \zeta_{\triangle_{sh}} (h, z) , \quad \forall \, z \in \mathbb{C} ,
\end{align}
and taking
$\displaystyle \frac{d}{dz} \bigg|_{z=0}$ yields the variation formula
\begin{align} \label{vardet2}
- \frac{d}{ds} \zeta^\prime_{\triangle_{sh}} (0)\, =
 \, \zeta_{\triangle_{sh}} (h, 0)  .
\end{align}
Applying Theorem \ref{thmmain} to the conformal weights $\vp_s$ with dilaton $sh$,
and retaining only the even part  of the
expression in the right hand side, yields
\begin{align} \label{varK}
\begin{split}
- \frac{d}{ds} \zeta^\prime_{\triangle_{sh}} (0)\, =-\frac{\pi}{\tau_2}\, &
\,  \vp_0\left( h\big(s K_0(s\lmod)(\triangle (\frac{h}{2})) +
s^2 H_0(s\nabla_1, s\nabla_2)(\square_\Re(\frac{h}{2}))\big)\right) \\
 &- \, \vp_n(h) ,
 \end{split}
\end{align}
One has $\vp_n(h)=-\frac{d}{ds} \log \vp_0 (e^{-sh})$
and so by integration along the interval $0 \leq s \leq 1$ one obtains the following
{\em conformal variation formula}.
\begin{lem}\label{inside0}
Let $\vp$ be a conformal weight with dilaton
  $h = h^\ast \in A_\theta^\infty$. Then
\begin{align} \label{detK}
\log \Det^\prime (\perta)&= \,\log \Det^\prime \D \, + \, \log \vp (1) \\ \notag
& - \frac{\pi}{\tau_2} \int_0^1 \vp_0\left( h\big(s K_0(s\lmod)(\triangle (\log k)) +
s^2 H_0(s\nabla_1,s\nabla_2)(\square_\Re(\log k))\big)\right) ds
\end{align}
\end{lem}

We now show that this formula simplifies much further.
\begin{lem}\label{inside}
For $f(u)$ a function in Schwartz space one has
\begin{equation}\label{inside1}
    \vp_0(h f(\lmod)(a))=f(0)\vp_0(ha)\qqq a \in A_\theta^\infty.
\end{equation}
\end{lem}
\proof The derivation $\lmod$ is given by the commutator with $-h$ and $\sigma_t=\modu^{-it}$
fixes $h$ and preserves the trace $\vp_0$. Thus writing $f$ as a Fourier transform
$$
\vp_0(h f(\lmod)(a))=\int g(t)\vp_0(h \sigma_t(a))dt=
\int g(t)\vp_0( \sigma_t(ha))dt=f(0)\vp_0(ha).
$$
\endproof
We thus get that
$$
\vp_0\left( h\big(s K_0(s\lmod)(\triangle (\frac{h}{2})\right)=\frac 12 K_0(0)
s\vp_0(h\triangle h)=\frac s6 \vp_0(h\triangle h)
$$
and integrating from $0$ to $1$ we obtain, under the hypothesis of Lemma \ref{inside0},
\begin{align}\label{inside2}
    \log \Det^\prime (\perta)&= \,\log \Det^\prime \D \, + \, \log \vp (1)- \frac{\pi}{12\tau_2} \vp_0(h\triangle h)  \\ \notag
& - \frac{\pi}{\tau_2} \int_0^1 \vp_0\left( h
s^2 H_0(s\nabla_1,s\nabla_2)(\square_\Re(\log k))\right) ds
\end{align}
Let us now simplify the last term of \eqref{inside2}. By construction the expression
$S=\square_\Re(\log k)$ can be expressed as a linear combination of squares of elements of $A_\theta^\infty$. Thus we really need to understand the quadratic form
\begin{equation}\label{quadform}
    Q(x)=\int_0^1 \vp_0\left( h
 H_0(s\nabla_1,s\nabla_2)(xx)\right)s^2 ds
\end{equation}
where writing $H_0$ as a Fourier transform
\begin{equation}\label{fourierh}
    H_0(u,v)=\int g(a,b)e^{-i(au+bv)}da db
\end{equation}
one has
$$
H_0(s\nabla_1,s\nabla_2)(xx)=\int g(a,b)\sigma_{sa}(x)\sigma_{sb}(x) da db.
$$
\begin{lem}\label{doublelem}
Let $k(u,v)$ be a Schwartz function, such that $k(v,u)=-k(u,v)$,  then
\begin{equation}\label{dlem1}
    \vp_0(h k(\nabla_1,\nabla_2)(xx))=-\frac 12\vp_0(\ell(\lmod)(x) x)
\end{equation}
where the function $\ell$ is given by
\begin{equation}\label{LL}
    \ell(u)=u \,k(u,-u)
\end{equation}
\end{lem}
\proof
Let us define the function of two variables
\begin{equation}\label{innerf}
    w(a,b)=\vp_0(h \sigma_{a}(x)\sigma_{b}(x))
\end{equation}
so that, with $k(u,v)=\int g(a,b)e^{-i(au+bv)}da db$ one has
\begin{equation}\label{quadform1}
\vp_0(h k(\nabla_1,\nabla_2)(xx))= \int g(a,b) w( a,b)  da db
\end{equation}
One has, using $\sigma_c(h)=h$ for all $c$,
\begin{itemize}
  \item $w(a+c,b+c)=w(a,b)$ for all $c$.
  \item $w(a,b)-w(b,a)=-\vp_0(\lmod(\sigma_{a}(x))\sigma_{b}(x))$
\end{itemize}
where the last equality follows, using the trace property of $\vp_0$, from
$$
\vp_0(h \sigma_{a}(x)\sigma_{b}(x)-h\sigma_{b}(x)\sigma_{a}(x))=
\vp_0([h, \sigma_{a}(x)]\sigma_{b}(x))\,.
$$
Thus one gets, using the antisymmetry of $k(u,v)$ and $g(a,b)$,
$$
\vp_0(h k(\nabla_1,\nabla_2)(xx))=-\frac 12  \int g(a,b) \vp_0(\lmod(\sigma_{a-b}(x)) x)  da db
$$
Moreover
$$
k(u,-u)=\int g(a,b)e^{-i(au-bu)}da db
$$
and one gets
$$
\vp_0(h k(\nabla_1,\nabla_2)(xx))=-\frac 12 \vp_0(\lmod(k(\lmod,-\lmod)(x)) x)
$$
which is the required equality.
\endproof
We used the hypothesis that $k$ is a Schwartz function in order to use freely the Fourier transform but since the spectrum of the operator $\lmod$ is bounded this hypothesis is not needed as long as we deal with smooth functions.
It follows using $k(u,v)=H_0(su,sv)$  that
\begin{equation}\label{dlem2}
    \vp_0(h H_0(s\nabla_1,s\nabla_2)(xx))=\frac 1s\vp_0(L_0(s\lmod)(x) x)
\end{equation}
where the function $L_0$ is given by
\begin{equation}\label{LL1}
    L_0(u)=-\frac 12 u H_0(u,-u)
\end{equation}
We now need to compute the integral in the variable $s$ of \eqref{quadform}.
\begin{lem}\label{compute}
\begin{equation}\label{qcompute}
    Q(x)=\vp_0(K_2(\lmod)(x)x)
\end{equation}
where the function $K_2$ is given by
\begin{equation}\label{k2}
    K_2(v)=\frac{1}{3}+\frac{4}{v^2}-\frac{2 {\rm coth}(\frac{v}{2})}{v}
\end{equation}
\end{lem}
\proof One has, by \eqref{valuel1},
$$
L_0(u)=\frac{2}{3}-\frac{2 {\rm coth}\left(\frac{u}{2}\right)}{u}+\sinh\left(\frac{u}{2}\right)^{-2}
$$
Thus, using \eqref{dlem2},
$$
Q(x)=\int_0^1 \vp_0\left( h
 H_0(s\nabla_1,s\nabla_2)(xx)\right)s^2 ds=\int_0^1 \vp_0(L_0(s\lmod)(x) x)s ds
$$
which gives \eqref{qcompute} for
$$
K_2(v)=\int_0^1 L_0(sv)sds=v^{-2}\int_0^v uL_0(u)du,
$$
 and one checks that this agrees with \eqref{k2} by showing that
 the derivative of $v^2 K_2(v)$ is $vL_0(v)$.
\endproof
We can thus simplify \eqref{inside2} and obtain, using $\log k=\frac h2$,
\begin{equation}\label{inside3}
   \log \Det^\prime (\perta)= \,\log \Det^\prime \D \, + \, \log \vp (1)- \frac{\pi}{12\tau_2} \vp_0(h\triangle h)   - \frac{\pi}{4\tau_2}  \vp_0\left(
K_2(\nabla_1 )(\square_\Re(h))\right)
\end{equation}

\begin{rem} {\rm Since $\triangle$ is isospectral to the Dolbeault-Laplacian on $\T^2$
whose spectrum is the set of $|n+m\tau|^2$ for $n,m\in \Z$, $(n,m)\neq (0,0)$,
$\,\log \Det^\prime \D$ remains the same as for the ordinary $2$-torus,
and is computed by the Kronecker limit formula, as in \cite[Theorem 4.1]{RS2};
explicitly,
\begin{align*}
 \log \Det^\prime \D   \, =- \frac{d}{ds} \big|_{s=0} \sum_{(n,m)\neq (0,0)}|n+m\tau|^{-2s}=\, \newval ,
\end{align*}
where  $\eta$ is the Dedekind eta function
\begin{align*}
\eta (\tau) \, = \, e^{\frac{\pi \, i}{12} \, \tau} \, \prod_{n>0} \left(1-e^{2\pi i n \tau} \right) .
\end{align*}
}
\end{rem}
 \bigskip

Recalling that  the logarithm of analytic torsion for a complex curve~\cite{RS2} is given by
  $$
\frac{1}{2} \sum_{q=0}^1 (-1)^q \log \Det^\prime (\triangle^{(q,0)}) \, =
\,- \frac{1}{2} \log \Det^\prime (\pertb) \,  ,
 $$
it is perhaps not surprising that the Osgood-Phillips-Sarnak functional~\cite[\S 2.1]{OPS}
involves the negative of the log-determinant. By analogy, we define the
translation invariant functional $F$
on the space of selfadjoint elements of $A^\infty_\theta$ by the formula
\begin{align} \label{Func1}
F (h): = \,- \log \Det^\prime (\perta) + \log \vp (1)
 \, = \,- \log \Det^\prime \left(e^{\frac{h}{2}} \D e^{\frac{h}{2}}\right) + \log \vp_0 (e^{-h}) .
\end{align}
Its invariance under rescaling is due to the fact that $\zeta_{\D_\vp}(0) = -1$, \cf Theorem
  \ref{thmconfindlap}.  Indeed, since $\quad \triangle_{h+c} = e^c \triangle_h$ for any
  $\, c \in \R$, one has
  \begin{align*}
\zeta_{\triangle_{h+c}} (z) \, = \, e^{-cz} \, \zeta_{\triangle_h} (z),
\end{align*}
it follows that
\begin{align*}
\begin{split}
F (h+c) &= \, \zeta^\prime_{\triangle_{h+c}} (0) \, + \, \log \vp_0 (e^{-h-c}) \, = \,
 - c \, \zeta_{\triangle_{h}} (0)\, + \, \zeta^\prime_{\triangle_{h}} (0) \,  \\
 & \,+ \log \vp_0 (e^{-c})  \, + \, \log \vp_0 (e^{-h}) \, = \, F(h) .
\end{split}
\end{align*}

\begin{thm} \label{thmmain3}
 The functional $F(h)$ has the expression
\begin{equation} \label{Func2}
F (h)= \,-\newval
+ \frac{\pi}{4\tau_2}  \vp_0\left(
(K_2-\frac 13)(\nabla_1 )(\square_\Re(h))\right) .
\end{equation}
One has $F(h)\geq F(0)$ for all $h$ and equality holds if and only if
$h$ is a scalar.
\end{thm}

\begin{proof}  One obtains using \eqref{inside3} and the classical value for $h=0$
\begin{eqnarray} \label{Func2bis}
F (h)&=& \,-\newval
+\frac{\pi}{12\tau_2} \vp_0(h\triangle h) \\ \notag
&+& \frac{\pi}{4\tau_2}  \vp_0\left(
K_2(\nabla_1 )(\square_\Re(h))\right) .
\end{eqnarray}
Using integration-by-parts with respect to the derivation $\delta^\ast$, one sees that
\begin{align} \label{positive}
\vp_0\big(h\,  \triangle(h)\big)  \, = \, \vp_0\big(h\, \delta^\ast (\delta(h))\big)
\, = \, - \vp_0\big(\delta^\ast (h) \, \delta(h)\big) \, = \,
 \vp_0\big(\delta(h)^\ast  \, \delta(h)\big) .
\end{align}
Note that the skew term
$$
\vp_0(-i\tau_2\delta_2(h)^*\delta_1(h)+i\tau_2 \delta_1(h)^*\delta_2(h))
=-i\tau_2\vp_0([\delta_1(h),\delta_2(h)])=0
$$
vanishes. Thus one has
\begin{equation}\label{deldel}
 \vp_0\big(h\,  \triangle(h)\big)=-\vp_0(\square_\Re(h)).
\end{equation}
This, together with \eqref{Func2bis} gives \eqref{Func2}.

\begin{figure}
\begin{center}
\includegraphics[scale=0.7]{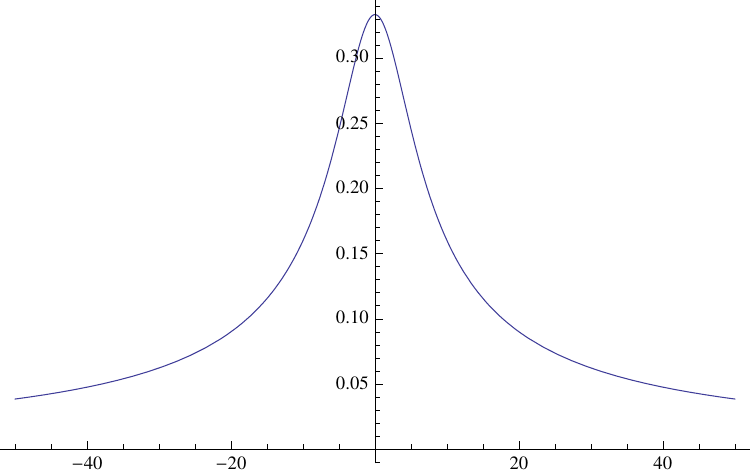}
\caption{Graph of the  function $\frac 13-K_2$, even and positive.}\label{K2function}
\end{center}
\end{figure}

In the Hilbert space $\cH_0$ with inner product \eqref{inprod}, the operator $\lmod$ which is given by $a\mapsto -[h,a]$ is bounded and self-adjoint since
$$
\langle \lmod(a),b\rangle =-\vp_0(b^*[h,a])=
\vp_0((-[h,b])^*a)=\langle a, \lmod(b)\rangle
$$
Since the function $\frac 13-K_2$ is even and strictly positive, it follows that the equality
\begin{equation}\label{scalk}
    \langle a,b\rangle_k=\langle (\frac 13-K_2)(\lmod)(a),b\rangle
\end{equation}
defines a non degenerate positive symmetric inner product on $\cH_0$. Now by \eqref{Func2}, and the skew adjointness $\delta_j(h)^*=-\delta_j(h)$ one has
$$
 \vp_0\left(
(K_2-\frac 13)(\nabla_1 )(\square_\Re(h))\right)=
\langle \delta_1(h),\delta_1(h)\rangle_k + 2 \Re(\tau) \langle \delta_1(h),\delta_2(h)\rangle_k+ |\tau|^2 \langle \delta_2(h),\delta_2(h)\rangle_k
$$
and thus
\begin{equation}\label{pospos}
F (h)= F(0)
+ \frac{\pi}{4\tau_2}\langle \delta(h),\delta(h)\rangle_k
\end{equation}
This shows that $F(h)\geq F(0)$ and moreover the equality holds only if $\delta(h)=0$ which implies that $h$ is a constant.
\end{proof}

 \begin{rem}{\em Note that
 the first term in the right hand side of \eqref{Func2bis} matches the
 expression of the functional in the commutative case (\cf \cite[\S 3]{OPS}),
while the second term is of a purely modular nature and is highly non-linear.}
 \end{rem}

 \bigskip

 \subsection{Gradient flow and scalar curvature} \hfill \medskip

Identifying the tangent space to the
dilatons with the selfadjoint elements of $A_\theta^\infty$, and
using the  inner product given by $\vp_0$, we
shall define the gradient of the
functional $F$ by means of the G\^ateaux differential:
\begin{align} \label{gradF}
\vp_0 (a\,  \grad_h F  ) \, = \,
\langle \grad_h F , a \rangle \, = \, \frac{d}{d\ve}\big|_{\ve=0}F(h + \ve a) , \qquad
\forall \, a = a^\ast \in A_\theta^\infty ,
\end{align}

\begin{thm}  \label{gradthm}
The gradient of $F$ is given by the expression
 \begin{equation} \label{gradF2}
 \grad_h F = \frac{\pi}{4\tau_2}\left( \tilde K_0(\lmod)(\lapa(h))+\frac 12\tilde H_0(\nabla_1,\nabla_2)(\square_\Re (h))\right)
 \end{equation}
 where the functions $\tilde K_0$ and $\tilde H_0$ are directly related to $K_0$, $K_2$ and $H_0$ by
 \begin{equation}\label{tildekh}
  \tilde K_0(s)=4\frac{\sinh(s/2)}{s} K_0(s)= -2\left(K_2(s)-\frac 13\right)
 \end{equation}
 and
 \begin{equation}\label{tildekh1}
  \tilde H_0(s,t))=4\frac{\sinh((s+t)/2)}{s+t} H_0(s,t).
 \end{equation}
\end{thm}

 \begin{proof}
Consider  the $1$-parameter family of operators
 \begin{equation*}
 \ttd \, =\,e^{ \frac{h+\ve a}{2} }\,\triangle\,e^{\frac{h+\ve a}{2} }       , \qquad  \ve > 0 .
\end{equation*}
The Duhamel (also the expansional) formula implies that
  \begin{align*}
 \frac{d}{d\ve} \big|_{\ve=0} e^{\frac{h+\ve a}{2}}   \,= \,
 \frac{1}{2}\, \int_0^1 e^{\frac{u h}{2}} \, a \, e^{\frac{(1-u)h}{2}} \, du ,
 \end{align*}
 hence
   \begin{align*}
   \begin{split}
 \frac{d}{d\ve} \big|_{\ve=0} \ttd   \,= \,
\frac{1}{2} \int_0^1 e^{\frac{u h}{2}} \, a \, e^{-\frac{uh}{2}} \, du \,
  \D_h \, + \, \frac{1}{2} \D_h \, \int_0^1 e^{-\frac{(1-u) h}{2}} \, a \, e^{\frac{(1-u)h}{2}} \, du .
 \end{split}
 \end{align*}

It follows that
\begin{align}\label{computederiv}
   \begin{split}
&\frac{d}{d \ve}  \big|_{\ve=0} \Tr \left(e^{-t  \ttd } \right) \,= \,
- t \, \Tr \left(\frac{d}{d\ve} \big|_{\ve=0} (\ttd) \,e^{-t \D_h} \right) \,=  \\
&= \,- \frac{t}{2} \, \Tr \left(\int_0^1 e^{\frac{u h}{2}} \, a \, e^{-\frac{uh}{2}}\, du \,
\D_h \, e^{-t  \D_h }  \right) \,
\,- \frac{t}{2} \, \Tr \left(\D_h \,\int_0^1 e^{-\frac{(1-u) h}{2}} \, a \, e^{\frac{(1-u)h}{2}}\, du \,
e^{-t  \D_h } \right) \\
&= \,- \frac{t}{2}\, \Tr \left(\D_h \,e^{-t  \D_h } \,\int_0^1 e^{\frac{u h}{2}} \, a \, e^{-\frac{uh}{2}} \, du\right) \, - \,  \frac{t}{2}\,
\Tr \left(e^{-t  \D_h } \,\D_h \,\int_0^1 e^{-\frac{u h}{2}} \, a \, e^{\frac{uh}{2}} \, du\right) \\
&= \,- \frac{t}{2}\,
\Tr \left(\D_h \,e^{-t  \D_h } \, \int_{-1}^1 e^{\frac{u h}{2}} \, a \, e^{-\frac{uh}{2}} \, du\right)
\, = \,\frac{t}{2} \, \frac{d}{dt}\, \Tr \left(e^{-t  \D_h } \,\int_{-1}^1 e^{\frac{u h}{2}} \, a \, e^{-\frac{uh}{2}} \, du\right) .
  \end{split}
 \end{align}

Differentiating at $\ve = 0$ the $1$-parameter family of the zeta functions
\begin{align*}
\zeta_{\ttd} (z) = \Tr \left( (1-P_\ve)\ttd^{-z} \right) , \quad  \Re z >  p ,
\end{align*}
and taking into account the above identity, one obtains for $\Re z >  p\,$,
\begin{align*}
\begin{split}
\frac{d}{d\ve} \big|_{\ve=0} \zeta_{\ttd} (z)&= \, \frac{1}{\Gamma (z)} \int_0^\infty t^{z-1}
 \,\frac{d}{d\ve} \big|_{\ve=0} \Tr \left(e^{-t \ttd} (1-P_\ve)\right) \, dt  \\
 &=  \, \frac{1}{\Gamma (z)} \int_0^\infty t^{z-1}
 \,\frac{d}{d\ve} \big|_{\ve=0} \Tr \left(e^{-t \ttd}\right) \, dt  \\
  &=  \,\frac{1}{2\, \Gamma (z)} \int_0^\infty t^{z} \, \frac{d}{dt}
 \,\Tr \left( e^{-t  \D_h } \,\int_{-1}^1 e^{\frac{u h}{2}} \, a \, e^{-\frac{uh}{2}}  \, du \right) \, dt \\
&= \frac{1}{2\, \Gamma (z)} \, t^{z}  \, \Tr \left(e^{-t \D_h} (1-P_h)\,
\int_{-1}^1 e^{\frac{u h}{2}} \, a \,e^{-\frac{uh}{2}}  \, du \right) \bigg|_0^\infty \\
&- \, \frac{z}{2\, \Gamma (z)} \int_0^\infty t^{z-1} \,  \Tr\left(h e^{-t \D_h} (1-P_{h})
\int_{-1}^1 e^{\frac{u h}{2}} \, a \,e^{-\frac{uh}{2}} \, du \right)\, dt \\
 &=\, -   \frac{z}{2\, \Gamma (z)} \int_0^\infty t^{z-1} \,  \Tr\left(e^{-t \D_h} (1-P_{h})
\int_{-1}^1 e^{\frac{u h}{2}} \, a \,e^{-\frac{uh}{2}}  \, du \right)\, dt \\
&= \, - \frac{z}{2} \, \zeta_{\D_h} \left( \int_{-1}^1 e^{\frac{u h}{2}} \, a \, e^{-\frac{uh}{2}}  \, du ,
\, z\right).
  \end{split}
\end{align*}
Now taking the derivative
$\displaystyle \frac{d}{dz} \bigg|_{z=0}$ gives
\begin{align*}
\begin{split}
- \,\frac{d}{d\ve} \bigg|_{\ve=0} \zeta^\prime_{\ttd} (0)&= \,  \frac{1}{2}\,
\zeta_{\D_h} \left(\int_{-1}^1 e^{\frac{u h}{2}} \, a \,e^{-\frac{uh}{2}}  \, du , \, 0\right) \\
&=  \frac{1}{2}\,{\rm a}_2 \left(\int_{-1}^1 e^{\frac{u h}{2}} \, a \, e^{-\frac{uh}{2}} \, du , \D_h\right)  -
 \frac{1}{2}\,\Tr \left(P_h \int_{-1}^1 e^{\frac{u h}{2}} \, a \, e^{-\frac{uh}{2}} \, du \right) \\
&=  \frac{1}{2}\,{\rm a}_2 \left(\int_{-1}^1 e^{\frac{u h}{2}} \, a \, e^{-\frac{uh}{2}} \, du , \D_h\right)  -
 \frac{1}{2}\,\frac{\vp_0\left(\int_{-1}^1 e^{\frac{u h}{2}} \, a \, e^{-\frac{uh}{2}} \, du \, e^{-h}\right)}{\vp_0 (e^{-h})} \\
&= \frac{1}{2}\,{\rm a}_2 \left(\int_{-1}^1 e^{\frac{u h}{2}} \, a \, e^{-\frac{uh}{2}} \, du , \D_h\right)  -
 \, \frac{\vp_0 (a e^{-h} )}{\vp_0 (e^{-h})}.
\end{split}
\end{align*}

 The calculation of the derivative corresponding to the area term is very easy.
 Indeed, since $\vp_0$ is a trace, one simply has
 \begin{align*}
 \frac{d}{d\ve}\big|_{\ve=0} \vp_0 (e^{-h-\ve a}) \, = \, - \vp_0 (a e^{-h}) ,
\end{align*}
 hence
  \begin{align*}
  \frac{d}{d\ve}\big|_{\ve=0} \log \vp_0 (e^{-h-\ve a})
= - \frac{\vp_0 (a e^{-h})}{ \vp_0 (e^{-h})} .
\end{align*}

Summing up and recalling that by its very definition, \cf \eqref{Func1},
 \begin{align*}
F (h+\ve a):\, = \,
 - \log \Det^\prime \left(e^{\frac{h+\ve a}{2}} \D e^{\frac{h+\ve a}{2}}\right) +
  \log \vp_0 (e^{-h-\ve a}) ,
\end{align*}
one concludes that
 \begin{align} \label{gradF3}
\vp_0 (a\,  \grad_h F ) \, = \, - \frac{1}{2}\, \,
{\rm a}_2 \left(\int_{-1}^1 e^{\frac{u h}{2}} \, a \, e^{-\frac{uh}{2}} \, du , \,  \D_h\right) .
 \end{align}
 To obtain the claimed expression of the gradient we
 appeal to Theorem \ref{thmmain}, move one of the exponential factors
 under  the trace $\vp_0$, and get
 $$
 \grad_h F  \, = \frac{\pi}{2\tau_2} \int_{-1}^1 e^{\frac{u h}{2}}\left(K_0(\lmod)(\lapa(\log k))+H_0(\nabla_1,\nabla_2)(\square_\Re (\log k)\right)
 e^{-\frac{uh}{2}} \, du
 $$
 One has
 $$
 \lmod(x)=-[h,x]\,, \ \ \int_{-1}^1 e^{\frac{u \lmod}{2}}du=4\frac{\sinh(\lmod/2)}{\lmod}
 $$
 thus the function $K_0$ gets multiplied by $4\sinh(s/2)/s$ and becomes
 $$
 \tilde K_0(s)=4\frac{\sinh(s/2)}{s} K_0(s)=\frac{4 \left(2+e^s (-2+s)+s\right)}{\left(-1+e^s\right) s^2}
 =4\left(\frac{{\rm coth}(s/2)}{s}-2 s^{-2}\right)
 $$
 similarly, one has
 $$
 \int_{-1}^1 e^{\frac{u (\nabla_1 + \nabla_2)}{2}}du=4\frac{\sinh((\nabla_1 + \nabla_2)/2)}{(\nabla_1 + \nabla_2)}
 $$
 which gives \eqref{tildekh1}.
 \end{proof}

  \begin{rem}\label{otherlap}{\em One may wonder if there is a similar manner of
  using   the local formula for $\pertc$ of Theorem \ref{thmmain}
  to handle
  $\log \Det^\prime (\pertc)   =  - \zeta^\prime_{\pertc} (0)$ .  To this end, we start
from the equality $\pertc = \delta^* k^2 \delta$ and the $1$-parameter family
 \begin{equation*}
 \pertd := \,\delta^*\,\,e^{h + \ve a}  \,  \delta   , \qquad  \epsilon \in \R .
\end{equation*}
Since
 \begin{align*}
 \frac{d}{d\ve} \big|_{\ve=0} e^{h+\ve a}   \,= \,
 \int_0^1 e^{u h} \, a \, e^{(1-u)h} \, du ,
 \end{align*}
one has
   \begin{align*}
   \begin{split}
 \frac{d}{d\ve} \big|_{\ve=0} \pertd    \,= \,
 \int_0^1 \delta^\star \, e^{u h} \, a \, e^{(1-u)h} \,  \delta  \, du .
 \end{split}
 \end{align*}
 The Duhamel formula
  \begin{equation*}
 \frac{d}{d\ve} \big|_{\ve=0} e^{-t \pertd}   \, = \, -t \int_0^1 e^{-vt \pertc} \,
\frac{d}{d\ve} \big|_{\ve=0} (\pertd)  \, e^{- (1-v)t \pertc} \, dv .
\end{equation*}
implies
\begin{align*}
   \begin{split}
&\frac{d}{d \ve}  \big|_{\ve=0} \Tr \left(e^{-t  \pertd } \right) \,= \,
- t \, \Tr \left(\frac{d}{d\ve} \big|_{\ve=0} (\pertd) \,e^{-t \pertc} \right) \,=  \\
&= -t \int_0^1\Tr \left( \delta^\star \, e^{u h} \, a \, e^{(1-u)h} \,  \delta
 \,e^{-t \pertc} \right)  \, du  \,= \,
  -t \int_0^1\Tr \left( e^{u h} \, a \, e^{(1-u)h} \,  \delta
 \,e^{-t \pertc}  \delta^\star \right)  \, du \\
 &=\, -t \int_0^1\Tr \left((e^{\frac{(2u-1) h}{2}} \, a \, e^{\frac{(1- 2u)h}{2}}) \, (e^{\frac{h}{2}} \,
   \delta
 \,e^{-t \pertc}  \delta^\star\, e^{\frac{h}{2}}) \right)  \, du \\
  &=\, -\frac{t}{2} \int_{-1}^1\Tr \left((e^{\frac{v h}{2}} \, a \, e^{-\frac{vh}{2}}) \,
 (e^{\frac{h}{2}} \,
   \delta
 \,e^{-t \pertc}  \delta^\star\, e^{\frac{h}{2}}) \right)  \, dv .
   \end{split}
 \end{align*}
Notice now that
\begin{align*}
e^{\frac{h}{2}} \,
   \delta
 \,e^{-t \pertc}  \delta^\star\, e^{\frac{h}{2}}  \, = \,
e^{-t \perta} \, \perta  ,
\end{align*}
reverting us to the expression \eqref{computederiv},
 which involves the local formula for $\perta$ and not for $\pertc$. Thus,
 one ends up by merely recovering the same expression \eqref{gradF2} of the gradient.

% A possible strategy would be to put the term $\delta^\star\, e^{\frac{h}{2}}$ in front using the % trace property and permute this term with $(e^{\frac{v h}{2}} \, a \, e^{-\frac{vh}{2}})$ but we % shall not pursue this here.
  }
 \end{rem}
 \bigskip

  Replacing in the definition \eqref{gradF} the fixed inner product
  by the running one, one could alternatively define the gradient by
  means of the running inner product, as follows:
  \begin{align} \label{newgradF}
\langle \Grad_h F , a \rangle_\vp \, = \, \vp_0 (a\,  \Grad_h F \, e^{-h} )
 \, = \, \frac{d}{d\ve}\big|_{\ve=0}F(h + \ve a) , \qquad
\forall \, a = a^\ast \in A_\theta^\infty .
\end{align}
Then Theorem \ref{gradthm}  gives
\begin{equation} \label{gradF4}
\Grad_h F  \, = \,\frac{\pi}{4\tau_2}\left( \tilde K_0(\lmod)(\lapa(h))+\frac 12\tilde H_0(\nabla_1,\nabla_2)(\square_\Re (h))\right)\, e^h.
 \end{equation}
  Based on the analogy with the standard torus (\comp \cite[\S 3, (3.8)]{OPS}), the
right hand side
\begin{equation} \label{scalar}
K_\vp \, = \,\frac{\pi}{4\tau_2}\left( \tilde K_0(\lmod)(\lapa(h))+\frac 12\tilde H_0(\nabla_1,\nabla_2)(\square_\Re (h))\right)\, e^h
 \end{equation}
can be taken as the appropriate
 definition of the \textbf{scalar curvature} $K_\vp$
of the conformal metric on the noncommutative torus associated to the
given dilaton. The evolution equation for the conformal factor $-h$ becomes
\begin{align} \label{gradflow2}
 \frac{\partial h}{\partial t} \, = \, K_\vp.
\end{align}

 Unlike the commutative case, the corresponding flow of inner products is not
 given by the same differential equation. Denoting by $g_\vp$ the Hermitian form
\begin{align*}
g_\vp (a, b) := \, \langle a, b \rangle_\varphi \, =\,  \varphi (b^\ast  a) \, = \,
 \vp_0 (b^\ast  a e^{-h}) , \quad \forall \, a, b \in A_{\theta} ,
\end{align*}
one has
\begin{align*}
\begin{split}
 \frac{\partial g_\vp (a, b)}{\partial t}& = \,  \vp_0 \left(b^\ast  a
  \frac{\partial e^{-h}}{\partial t}\right) \, = \, - \vp_0 \left(b^\ast  a
   \int_0^1e^{-uh}  \frac{\partial h}{\partial t} e^{(u-1)h}\, du \right) \\
& =  \,- \vp_0 \left(b^\ast  a   \int_0^1e^{-uh}  K_\vp e^{uh} \, du\, e^{-h} \right) \,
= \,- \langle a \int_0^1e^{-uh} K_\vp e^{uh}\, du \, , b \rangle_\varphi \, .
\end{split}
\end{align*}
Denoting by $R_\vp$ the resulting Hermitian form
\begin{align} \label{Ricci}
R_\vp ( a, b) \, =\,    \langle a \frac{e^{\lmod} -1}{\lmod}( K_\vp) , b \rangle_\varphi \,
 , \quad \forall \, a, b \in A_{\theta} ,
\end{align}
we conclude that the metric associated to $\vp$
has evolution equation
\begin{align} \label{Ricciflow}
 \frac{\partial g_\vp}{\partial t} \, = \, - R_\vp .
\end{align}
Since the average curvature 
$$
\vp (K_\vp) \, =\, \vp_0 (\grad_h F)\, = \, \frac{d}{d\ve}\big|_{\ve=0}F(h + \ve)  \, = \, 0 ,
$$ 
the equation \eqref{Ricciflow} is
precisely the analogue of Hamilton's Ricci flow~\cite{Ham} for the standard torus.
This justifies viewing the Hermitian form $R_\vp$ as
the \textbf{Ricci curvature} of $\T^2_\theta$ endowed with the inner product
$g_\vp$.

\bigskip
 
\subsection{Functional relation between $\tilde K_0$ and $\tilde H_0$}\label{functrel}

We shall now explain how to perform the direct computation of the gradient of an expression of the form
$$
\vp_0\left(
G(\nabla_1 )(\square_\Re(h))\right)
$$
We shall then show how this, together with Theorem \ref{gradthm}, implies the following relation
between $\tilde K_0$ and $\tilde H_0$,
\begin{equation}\label{transforule}
  - \frac 12 \tilde H_0(s_1,s_2)=\frac{\tilde K_0(s_2)-\tilde K_0(s_1)}{s_1+s_2}+\frac{\tilde K_0(s_1+s_2)-\tilde K_0(s_2)}{s_1}-\frac{\tilde K_0(s_1+s_2)-\tilde K_0(s_1)}{s_2}
\end{equation}
This relation can then be checked directly and  gives the
 following decomposition
 $$
- \frac 18 \tilde H_0(s,t) =\frac{2 (s-t)}{s t (s+t)^2}$$
$$+\frac{{\rm coth}\left(\frac{s}{2}\right)}{s t}-\frac{{\rm coth}\left(\frac{s}{2}\right)}{s (s+t)}-\frac{{\rm coth}\left(\frac{t}{2}\right)}{s t}+\frac{{\rm coth}\left(\frac{t}{2}\right)}{t (s+t)}+\frac{{\rm coth}\left(\frac{s+t}{2}\right)}{s (s+t)}-\frac{{\rm coth}\left(\frac{s+t}{2}\right)}{t (s+t)}\,.
 $$
 The fact that \eqref{transforule} can be proven on a priori ground gives a handle on the complicated two variable functions which appear in Theorem \ref{thmmain}
since by \eqref{tildekh1} one can deduce the function $H_0$ from $\tilde H_0$. Note moreover that the function $\frac 18\tilde K_0$ is the generating function for Bernoulli numbers since one has
\begin{equation}\label{bernou}
   \frac 18\tilde K_0(u)=\sum_1^\infty \frac{B_{2n}}{(2n)!}u^{2n-2}\,.
\end{equation}

\begin{thm}\label{directgrad}
Let $G(u)$ be an even Schwartz function, then with
\begin{equation}\label{defnfunc}
   \Omega(h)=\vp_0\left(
G(\nabla_1 )(\square_\Re(h))\right)
\end{equation}
one has
\begin{equation}\label{firsttransfo0}
 \frac{d}{d\ve} \big|_{\ve=0} \Omega(h+\epsilon a)\,= \,
 -2\vp_0(a G(\lmod)(\lapa(h)))+
\vp_0(a \,\omega_G(\nabla_1,\nabla_2)(\square_\Re(h)))
 \end{equation}
 where the function $\omega_G(s,t)$ is given by
 \begin{equation}\label{firsttransfo1}
   \frac 12 \omega_G(s_1,s_2)=\frac{G(s_2)-G(s_1)}{s_1+s_2}+\frac{G(s_1+s_2)-G(s_2)}{s_1}-\frac{G(s_1+s_2)-G(s_1)}{s_2}
\end{equation}
\end{thm}
\proof It is enough to prove the statement with $\square_\Re(h)$ replaced by $\delta(h)\delta(h)$ and $\lapa(h)$ by $\delta^2(h)$ where $\delta$ is a derivation equal to $\delta_j$ or $\delta_1+\delta_2$. One has
\begin{align*}
   \begin{split}
&  \frac{d}{d\ve} \big|_{\ve=0} \vp_0\left(
G(\nabla_1 )(\delta(h)\delta(h)\right)= \,
\vp_0 \big(\frac{d}{d\ve} \big|_{\ve=0} G (\log \modu_{h + \ve a})
(\d(h)) \, \d(h)\big) \\
&+ \vp_0 \big(G (\log \modu_h)(\d(a)) \, \d(h)\big)
+ \vp_0 \big(G (\log \modu_h)(\d(h)) \, \d(a)\big) .
   \end{split}
 \end{align*}
The proof of Theorem \ref{directgrad} then follows from the equality between the last two terms and the following two general lemmas. \endproof
\begin{lem}\label{doubletransfo}
Let $G(u)$ be an even Schwartz function, then for any $x,a \in A_\theta^\infty$ one has
\begin{equation}\label{doubletransfo0}
 \frac{d}{d\ve} \big|_{\ve=0} \vp_0( G (\log \modu_{h + \ve a})(x) \, x )\,= \,
\vp_0(a H(\nabla_1,\nabla_2)(xx))
 \end{equation}
 where the function $H(s,t)$ is given by
 \begin{equation}\label{doubletransfo1}
    H(s_1,s_2)=2\frac{G(s_2)-G(s_1)}{s_1+s_2}
 \end{equation}
\end{lem}
\proof
 Using the Fourier transform
 $$
 G(v)=\int g(t)e^{-itv}dt
 $$
  and the equality
  $$
  \log \modu_{h + \ve a}=\log \modu -\epsilon {\rm ad}_a\,, \ \ {\rm ad}_a(z)=[a,z]\qqq z\in A_\theta^\infty,
  $$
   we write
 \begin{align*}
 G (\log \modu_{h + \ve a}) (x) \, = \, \int g(t) e^{-it\lmod+\epsilon it{\rm ad}_a} (x) dt .
  \end{align*}
Since
  \begin{align*}
 \frac{d}{d\ve} \big|_{\ve=0} e^{-it\lmod+\epsilon it{\rm ad}_a}   \,= \,
  \int_0^1 e^{-iu t \lmod} \, i t {\rm ad}_a \, e^{-i(1-u)t \lmod} \, du ,
 \end{align*}
 one obtains
 \begin{equation*} \frac{d}{d\ve} \big|_{\ve=0} \vp_0( G (\log \modu_{h + \ve a})(x) \, x )\,=
\int g(t) \int_0^1 \vp_0 \left( \sigma_{u t}( \, i t {\rm ad}_a \, \sigma_{(1-u)t}(x))x \right)\, du dt
 \end{equation*}
One has
$$
\vp_0 \left( \sigma_{u t}( \, i t {\rm ad}_a \, \sigma_{(1-u)t}(x))x \right)=
it\vp_0\left({\rm ad}_a ( \sigma_{(1-u)t}(x))\sigma_{-u t}(x)\right)$$ $$
=it\vp_0\left(a (\sigma_{(1-u)t}(x)\sigma_{-u t}(x)-\sigma_{-u t}(x)\sigma_{(1-u)t}(x)\right),
$$
which is of the form
$$
\vp_0(a \ell(\nabla_1,\nabla_2)(xx))\,, \ \ \ell(s_1,s_2)=e^{-i s_1(1-u)t-is_2(-u)t}
-e^{-i s_1(-u)t-is_2(1-u)t}
$$
and gives \eqref{doubletransfo0} for
$$
H(s_1,s_2)=\int g(t) \int_0^1 it \left(e^{-i s_1(1-u)t-is_2(-u)t}
-e^{-i s_1(-u)t-is_2(1-u)t}\right) dt du
$$
which gives
\begin{equation}\label{predoubletransfo1}
    H(s_1,s_2)=\int_0^1\big(G'(-us_1+(1-u)s_2)-G'((1-u)s_1-u s_2)\big)du.
 \end{equation}
 One then performs the integral to obtain \eqref{doubletransfo1}.
 \endproof
As a simple example we take $G(u)=u^2$. In that case one has
$$
\frac{d}{d\ve} \big|_{\ve=0} \vp_0( G (\log \modu_{h + \ve a})(x) \, x )\,=-\frac{d}{d\ve} \big|_{\ve=0} \vp_0([h+\epsilon a,x]^2)
$$
and the right hand side gives
$$
-\frac{d}{d\ve} \big|_{\ve=0} \vp_0([h+\epsilon a,x]^2)=-2\vp_0([a,x][h,x])=
\vp_0(a H(\nabla_1,\nabla_2)(xx))
$$
with $H(s_1,s_2)=2s_2-2s_1$.

Let us now consider the term
$$
\vp_0 \big(G (\log \modu_h)(\delta(a)) \, \delta(h)\big) \, .
$$
We need to integrate by parts, which is achieved as follows.

\begin{lem}\label{doubletransfobis}
Let $G(u)$ be a Schwartz function, then for any $a \in A_\theta^\infty$ one has
\begin{equation}\label{doubletransfo0bis}
  \vp_0 \big(G (\log \modu_h)(\delta(h)) \, \delta(a)\big)\,= \,
\vp_0(a L(\nabla_1,\nabla_2)(\delta(h)\delta(h)))-\vp_0(a
G (\log \modu_h)\delta^2(h)),
 \end{equation}
 where the function $L(s,t)$ is given by
 \begin{equation}\label{doubletransfo1bis}
    L(s_1,s_2)=\frac{G(s_1+s_2)-G(s_2)}{s_1}-\frac{G(s_1+s_2)-G(s_1)}{s_2}
 \end{equation}
\end{lem}
\proof One has for any $x\in A_\theta^\infty$ the equality
$$
\delta(\sigma_t(x))-\sigma_t(\delta(x))=it\int_0^1\sigma_{ut}({\rm ad}_{\delta(h)}(\sigma_{(1-u)t}(x)))du
$$
so that
$$
\delta(\sigma_t(x))-\sigma_t(\delta(x))=it\int_0^1\left(\sigma_{ut}(\delta(h))\sigma_t(x)-
\sigma_t(x)\sigma_{ut}(\delta(h))\right)du
$$
and taking $x=\delta(h)$ we get,
$$
\delta(\sigma_t(\delta(h)))-\sigma_t(\delta^2(h))=
L_t(\nabla_1,\nabla_2)(\delta(h)\delta(h))
$$
where
$$
L_t(s_1,s_2)=\frac{1-e^{-its_1}}{s_1}e^{-its_2}-\frac{1-e^{-its_2}}{s_2}e^{-its_1}
$$
Now writing $G(v)=\int e^{-itv}g(t)dt$ one gets
\begin{equation}\label{deltasigma}
\delta(G (\log \modu_h)(\delta(h)))=G (\log \modu_h)(\delta^2(h))-L(\nabla_1,\nabla_2)(\delta(h)\delta(h))
\end{equation}
where
$$
L(s_1,s_2)=-\int \left(\frac{1-e^{-its_1}}{s_1}e^{-its_2}-\frac{1-e^{-its_2}}{s_2}e^{-its_1}\right)g(t)dt
$$
which is the same as \eqref{doubletransfo1bis}.
One has, using integration by parts
$$
\vp_0 \big(G (\log \modu_h)(\delta(h)) \, \delta(a)\big)\,=
-\vp_0 \big(\delta(G (\log \modu_h)(\delta(h))) \,a\big)
$$
and using \eqref{deltasigma} one obtains \eqref{doubletransfo0bis}.
\endproof
\bigskip

%%%%%%%%%%%%%%%%%%%%%%%%%%

\section{Further remarks}

\subsection{Explicit examples} \hfill \medskip\label{sectexample}

As a concrete illustration, we shall compute the Ray-Singer determinant and
the scalar curvature for a class of conformal factors which 
exhibit interesting geometric features and have no classical
counterparts.
These are the dilatons associated to self-adjoint idempotents,
such as the Powers-Rieffel projections,
which exist in abundance in $A_\theta^\infty$. 
 In order to describe them it will be convenient to
 identify the $C^*$-subalgebra generated by $V$ with $C(S^1)$, or equivalently 
 to represent the algebra $A_\theta$ as the
crossed product of $C(S^1)$ by the irrational rotation, so that 
$$
UfU^*=f_\theta \qqq f \in C(S^1), \ \ f_\theta(x)=f(x-\theta)\qqq x\in \R/(2\pi \Z), \ \ V(x)=e^{ix} .
$$
A Powers-Rieffel projection~\cite{Ri} has the form
\begin{equation}\label{idemp}
    p=f_{-1} U^*+ f_0+ f_1 U\,, \ \ f_j\in C^\infty(S^1) ,
\end{equation}
where, assuming the functions $f_j$, $j=-1, 0, 1$, real valued, one has
\begin{align} \nonumber
   & f_{-1}(x)=f_1(x+\theta)\,, \ f_1(x)f_1(x-\theta)=0\,, \,
    f_1(x)(f_0(x)+f_0(x-\theta))=f_1(x), \\ \label{edemp1}
    & f_0(x)^2+f_1(x)^2+f_1(x+\theta)^2=f_0(x)\qqq x \in \R/(2\pi \Z) \, ;
\end{align}
moreover, one can choose $f_0$ taking values in $[0,1]$ and such that 
\begin{align} \label{edemp2}
\vp_0 (p) \, = \,\frac{1}{2\pi}\int_0^{2\pi} f_0(x)\,  dx \, = \, \theta .
\end{align}

We now fix a projection $p=p^*=p^2$ as above, and consider $1$-parameter
familiy of dilatons of the form
\begin{align} \label{pdilaton1}
h \, \equiv \, h (s) \, := \, s \, p \, + \, \rho (s) \, , \quad
\text{with} \quad \rho (s) \,:= \, \log \left(1 + (e^{-s} - 1)\, \theta \right) , \quad s \in \R \, .
\end{align}
The function $\rho$ is chosen so that the corresponding conformal weights
$$
\vp_s (x) \, :=\, \vp_0 \left(x e^{-s\, p - \rho(s)}\right) \, , \quad  x \in A_\theta^\infty \, ,
$$
are actually states; indeed,
 \begin{align} \label{pdilaton2}
\vp_s (1) \, = \, \vp_0 \left(e^{-s\, p - \rho(s)}\right) \, = \, 
e^{-\rho(s)}  \left(1 + (e^{-s} - 1) \vp_0 (p) \right) \, = \,  1 .
\end{align}

\begin{prop} \label{propexample}
With the above notation, the
Ray-Singer determinant of the Laplacian $\,\triangle_{\vp_s}$, $\, s \in \R$,
is given by the following closed formula:
 \begin{equation}\label{raysingerexpl}
\log \Det^\prime (\triangle_{\vp_s}) \, = \, \newval
-\frac{\pi}{8 \tau_2} \left(\a(p)+|\tau|^2 \b(p) \right) \, s^2 \tilde K_0(s) \, ,
\end{equation}
where
\begin{equation}\label{raysing}
   \a(p)=\frac{1}{\pi}\int_0^{2\pi}f_1(x)^2 dx\,, \quad
  \b(p)= \frac{1}{2\pi}\int_0^{2\pi}\left(f'_0(x)^2+2f'_1(x)^2\right) dx\,.
\end{equation}
\end{prop}

\proof Since $p^2=p$ one gets
\begin{equation}\label{decomp}
\delta_j(p)=p \delta_j(p)+\delta_j(p) p \, , \qquad j=1,2
\end{equation}
and with $\lmod$ the derivation implemented by $-h$ one has
\begin{equation}\label{eigen}
    \lmod(p\delta_j(h))=-s p\delta_j(h)\,, \ \  \lmod(\delta_j(h)p)=s \delta_j(h)p
\end{equation}
since $$[-h,p\delta_j(h)]=-s\left(p^2\delta_j(h)-p\delta_j(h)p\right)
=-s p\delta_j(h).$$
Thus the decomposition \eqref{decomp} gives $\delta_j(p)$ as a sum of eigenvectors for the eigenvalues $\pm s$ for the operator $\lmod$. Since the function $\tilde K_0$ is even we thus get
$$
\tilde K_0(\nabla_1 )(\delta_j(h)\delta_j(h))=
\tilde K_0(s)\left( p \delta_j(h)+\delta_j(h) p    \right)\delta_j(h)=
  s^2\tilde K_0(s)\delta_j(p)^2
$$
and
\begin{equation}\label{eigen1}
    \tilde K_0(\nabla_1 )(\square_\Re(h))= s^2\tilde K_0(s)\square_\Re(p)
\end{equation}
Using our formula for the variation of the log-determinant, \cf \eqref{inside3}, one obtains
\begin{equation}\label{raysingerbis}
\log \Det^\prime (\perta) \, = \,\newval
+\frac{\pi}{8\tau_2} s^2 \tilde K_0(s)\, \vp_0(\square_\Re(h)) .
\end{equation}
One has
\begin{align*}
\begin{split}
\vp_0(\delta_1(p)^2)&=-\frac{1}{2\pi}\int_0^{2\pi}\left(f_1(x)^2+f_1(x+\theta)^2\right) dx
=-\frac{1}{\pi}\int_0^{2\pi}f_1(x)^2 dx , \\
\vp_0(\delta_2(p)^2)&=-\frac{1}{2\pi}\int_0^{2\pi}\left(f'_0(x)^2+f'_1(x)^2+f'_1(x+\theta)^2\right) \\
&=-\frac{1}{2\pi}\int_0^{2\pi}\left(f'_0(x)^2+2f'_1(x)^2\right) dx , \\
\vp_0(\delta_1(p)\delta_2(p))&=-\frac{i}{2\pi}
\int_0^{2\pi}\left(f_1(x)f'_1(x) -f_1(x+\theta)f'_1(x+\theta) \right) dx=0 , \\
\vp_0(\delta_2(p)\delta_1(p))&=\frac{i}{2\pi}
\int_0^{2\pi}\left(f_1(x)f'_1(x) -f_1(x+\theta)f'_1(x+\theta) \right) dx=0 .
\end{split}
\end{align*}
Thus, the formula \eqref{raysingerbis} takes the form \eqref{raysingerexpl}.
\endproof
\medskip

Let us now compute the scalar curvature given by \eqref{scalar} for an arbitrary projection
$p \in A_\theta^\infty$. Since the normalization of the area plays no role in this local 
calculation, we now take $h = s\, p$. One has
$$
\lapa(h)=s \lapa(p)= s\left( p\lapa(p) p+ p\lapa(p)(1-p)+(1-p)\lapa(p)p
+(1-p)\lapa(p)(1-p)    \right)
$$
which gives a decomposition in eigenvectors for $\lmod$ with eigenvalues $0,-s,s,0$. It follows that
\begin{align} \label{K-term}
\begin{split}
\tilde K_0(\lmod)(\lapa(h))&= \, s \tilde K_0(s)(p\lapa(p)(1-p)+(1-p)\lapa(p)p)+
s\tilde K_0(0)(p\lapa(p) p \\
&\, +(1-p)\lapa(p)(1-p) )
\end{split}
\end{align}
One has moreover using the decomposition \eqref{decomp} and the vanishing of $\tilde H_0(s,s)$,
\begin{align} \label{H-term}
\frac 12\tilde H_0(\nabla_1,\nabla_2)(\square_\Re (h))=\frac 12 s^2\tilde H_0(s,-s)(1-2p)
\square_\Re (p)\,.
\end{align}
\smallskip

\begin{prop} \label{propexamplebis}
Let $p=p^*=p^2$ be any projection, $h=s\,p$, where $s\in\R$,
and $\vp(x)=\vp_0(x e^{-s\, p})$ the associated conformal weight.
The scalar curvature is given by the formula
\begin{equation}\label{finescal}
K_\vp= \frac{\pi s}{4\tau_2} \left(\tilde K_0(s)\lapa(p)+
\frac s2\partial_s \tilde K_0(s)\left(p\lapa(p) p+(1-p)\lapa(p)(1-p) \right)\right)e^h.
\end{equation}
\end{prop}
\proof The above discussion yields the formula
\begin{align}
\begin{split}\label{scalarexplpre}
K_\vp &= \frac{\pi s}{4\tau_2} \left(\tilde K_0(s)\left(p\lapa(p)(1-p)+
(1-p)\lapa(p)p\right)\right.\\&+
\left.\frac 23\left(p\lapa(p) p+(1-p)\lapa(p)(1-p) \right)+
\frac s2 \tilde H_0(s,-s)(1-2p)
\square_\Re (p)\right)e^h
\end{split}
\end{align}
But since $p^2=p$ one has the relation
\begin{equation}\label{relidem}
    2\square_\Re (p)=(1-p)\lapa(p)-\lapa(p)p
\end{equation}
which gives
\begin{equation}\label{relidembis}
    2(1-2p)\square_\Re (p)=(1-p)\lapa(p)-(1-2p)\lapa(p)p=p\lapa(p) p
    +(1-p)\lapa(p)(1-p) .
\end{equation}

Now,  one has the relation
\begin{equation}\label{special}
\frac s2 \tilde H_0(s,-s)=s\partial_s \tilde K_0(s)+2(\tilde K_0(s)-\tilde K_0(0)) ,
\end{equation}
which is a special case of \eqref{transforule}. We then use
$$
\left(p\lapa(p)(1-p)+(1-p)\lapa(p)p\right)+\left(p\lapa(p) p
+(1-p)\lapa(p)(1-p) \right)=\lapa(p)
$$
and simplify \eqref{scalarexplpre} to
$$
K_\vp= \frac{\pi s}{4\tau_2} \left(\tilde K_0(s)\lapa(p)+
\frac s2\partial_s \tilde K_0(s)\left(p\lapa(p) p+(1-p)\lapa(p)(1-p) \right)\right)e^h.
$$
which is the required equality.\endproof
\medskip

Note that each of the two separate terms  
$\, p\lapa(p) p+(1-e)\lapa(p)(1-p) $ and $\lapa(p)$
have vanishing integral under $\vp_0$, confirming the validity
of the Gauss--Bonnet formula.
 \medskip

A more striking fact is the `bending' along the ray of conformal factors $\, h_s = s\, p$
of the normalized scalar curvature
\begin{align} \label{nscalar}
\cK_{s}(p)\, :=\,  K_{\vp_s} \, e^{-h_s} \, .
\end{align} 
Classically, the normalized curvature is given by the Laplacian of the conformal factor,
and therefore it is homogeneous of degree $1$ in the scaling parameter.
In our case though, because $p \in A^\infty_\theta$ is an idempotent, $\cK_{s}(p)$
turns out to be bounded as a function of $\, s \in \R$. Indeed,
\begin{align*}
\cK_{s}(p)\, = \, \frac{\pi}{4\tau_2} \left(s \tilde K_0(s)\lapa(p)+
\frac{s^2}{2} \partial_s \tilde K_0(s)\left(p\lapa(p) p+(1-p)\lapa(p)(1-p) \right)\right) ,
\end{align*} 
with the odd functions 
\begin{align*} 
\frac{1}{4} s \tilde K_0(s) = 2 \left(\frac{1}{e^s - 1} - \frac 1 s\right) + 1, \quad
 \frac{1}{8}  s^2  \partial_s \tilde K_0(s)  = 
-\frac{s e^s}{(e^s - 1)^2}   -  \frac{1}{e^s -1} + \frac 2 s - \frac 1 2
\end{align*} 
evidently bounded. Moreover, one has
\begin{align*}
\lim_{s \rightarrow \pm \infty} \cK_{s}(p)\, =\, \pm \frac{\pi}{\tau_2} 
\left(\lapa(p) \, - \, \frac 12 \left(p\lapa(p) p+(1-p)\lapa(p)(1-p) \right)\right) .
 \end{align*}

\bigskip

\subsection{Intrinsic definition and normalization} \hfill \medskip\label{sectnorma}

In this section we explain how to reformulate the above scalar curvature in intrinsic terms involving only the even two-dimensional modular spectral triple $
(A_{\theta}^{\infty}, \cH, \dvp)$. Let us first relate the weight $\vp$ to the natural volume form associated to the modular spectral triple.
\begin{lem}\label{volume}
For any $a\in A_{\theta}$ one has
\begin{equation}\label{volume1}
    \cutint a \dvp^{-2}=\frac{2\pi}{\tau_2}\vp(a)\,, \ \ \cutint\gamma a \dvp^{-2}=0.
\end{equation}
\end{lem}
\proof We shall show that
\begin{equation}\label{volume2}
\cutint E a \dvp^{-2}=\frac{\pi}{\tau_2}\vp(a), \ \ E=\frac{1+\gamma}{2}.
\end{equation}
Note that since the kernel of $\dvp$ is finite dimensional we do not care about the lack of invertibility of $\dvp$ and define arbitrarily $\dvp^{-2}$ on the kernel, this does not affect the value of the residue. To prove \eqref{volume2} note that $E\dvp^2E=k\lapa k$ and one gets
$$
\cutint E a \dvp^{-2}=\cutint  a k^{-1}\lapa^{-1}k^{-1}=
\cutint k^{-1} a k^{-1}\lapa^{-1}=\lambda\vp_0(k^{-1} a k^{-1})=\lambda\vp_0(a k^{-2})
$$
where the constant $\lambda$ comes from the equality
$$
\cutint x\,\lapa^{-1}=\lambda\vp_0(x)\qqq x\in A_\theta
$$
which can be checked directly since both sides vanish on $U^nV^m$ for $(n,m)\neq (0,0)$. To compute $\lambda$ one just needs the residue at $s=1$ of the zeta function
$$
\Tr(\lapa^{-s})=\sum_{(n,m)\neq (0,0)}|n+m\tau|^{-2s}
$$
which gives $\pi/\tau_2$ and proves \eqref{volume2}. We shall not need the other part
but it follows from the above.\endproof

Using the canonical volume form of the spectral triple instead of $\vp$ in the definition of the scalar curvature \eqref{scalar} thus removes the unpleasant factor  $\pi/\tau_2$ in the above formulas.

\medskip

There is one more small adjustment needed to obtain an intrinsic definition; we have identified above the space of deformed modular spectral triples with the space of self-adjoint elements $h\in A_\theta^\infty$ and the tangent space at a point $h$ accordingly by linearity.

\begin{rem}\label{dimtwo} {\rm In \cite{CMbook}, Definition 1.147, the notion of scalar curvature functional was introduced  for spectral triples $(A,\cH,D)$ of dimension $4$ by the equality
\begin{equation}\label{dim4}
  \cR(a)\, = \,  \cutint a D^{-2}\qqq a \in A\,.
\end{equation}
The same formula with $D^{-(n-2)}$ instead of $D^{-2}$ works in dimension $n$ when $n\neq 2$
with a suitable normalization.
For $n=2$ the normalization factor has a pole and the analogue of \eqref{dim4} becomes,
\begin{equation}\label{dim2}
  \cR(a)=  {\rm a}_2(a,D^2)=\zeta_{D^2}(a,0)+\Tr(Pa)\qqq a \in A\,,
\end{equation}
where $P$ is the orthogonal projection on the kernel of $D$. What the present development shows in particular is that in the even case the above
general definition should be refined by using the chiral expression:
$$
 \cR_\gamma(a)\, = \,\cutint E a D^{-2}\,, \quad \text{where} \quad \ E=\frac{1+\gamma}{2}.
$$
A number of the above results extend naturally
to the general case of dimension two, and relate the variation
 under inner twisting of the Ray-Singer torsion of modular spectral triples
to their chiral scalar curvature.
}\end{rem}

\bigskip

%%%%%%%%%%%%%%%%%%%%%%%%%%%%%%%%%%%%%%%%%%%

\section{Symbolic calculations}\label{tech}

The computation follows the same lines as \cite{Paula}. The case of the operator $k\lapa k$ is treated at the symbol level as in  \cite{Paula}.  We first give the analogue of Lemma 6.1 of \cite{Paula} for the operator $\pertc= \delta^* k^2\delta$.

\begin{lem}
-- The operator $\pertc$ has symbol $\sigma (\pertc) = a_2 (\xi) + a_1
(\xi) + a_0 (\xi)$ where,
\begin{equation*}
a_2 = a_2 (\xi) =k^2 \left(\xi _1^2+2 \Re(\tau) \xi _1 \xi _2+|\tau|^2 \xi _2^2\right)
\end{equation*}
\begin{equation*}
a_1 = a_1 (\xi) =\left(k\delta _1(k)+\delta _1(k)k+\tau (k\delta _2(k)+\delta _2(k)k)\right) \left(\xi _1+\bar\tau \xi _2\right)
\end{equation*}
\begin{equation*}
a_0 = a_0 (\xi) = 0 .
\end{equation*}
\end{lem}
\proof This follows from the derivation property of $\delta^*=\delta_1+\tau \delta_2$ which gives
$$
\delta^* k^2\delta=k^2\delta^* \delta+ \delta^*(k^2)\delta,
$$
which is the decomposition of the operator as a sum of homogeneous terms.
\endproof
In general we recall the product formula within
the algebra of symbols, where with $\sigma(P)=\rho$, $\sigma(Q)=\rho'$, one has
\begin{equation*}
\sigma (PQ) \sim \sum_{\ell_j\geq 0} (1/(\ell_1! \, \ell_2!) \, [\partial_1^{\ell_1} \,
\partial_2^{\ell_2} (\rho (\xi)) \, \delta_1^{\ell_1} \,
\delta_2^{\ell_2} (\rho' (\xi))] \, .
\end{equation*}
One then proceeds as in \cite{Paula} for the inductive calculation of the inverse of the symbol of
$\pertc-\lambda$, using $\lambda$ as a symbol of order two and
\begin{equation}
\label{eq2} b_0 = b_0 (\xi) = (k^2  (\xi_1^2+2\Re(\tau)\xi_1\xi_2+|\tau|^2\xi_2^2) -\lambda)^{-1}
\end{equation}
and computing to order $-3$ in $\xi$ the product $b_0 \cdot ((a_2 -\lambda)
+ a_1 + a_0)$. By singling out terms of the appropriate degree $-1$
in $\xi$, one obtains
\begin{equation}
\label{eq3} b_1 = - (b_0 \, a_1 \, b_0 + \partial_i (b_0) \,
\delta_i (a_2) \, b_0) \, .
\end{equation}
Note that $b_0$ appears on the right in this formula, and we refer to \cite{Paula} for detailed explanations. In a similar fashion, collecting terms
of degree $-2$ in $\xi$    one obtains
\begin{eqnarray}
\label{eq4}
b_2 &= &- (b_0 \, a_0 \, b_0 + b_1 \, a_1 \, b_0 + \partial_i (b_0) \, \delta_i (a_1) \, b_0 \nonumber \\
&+ &\partial_i (b_1) \, \delta_i (a_2) \, b_0 + (1/2) \, \partial_i
\, \partial_j (b_0) \, \delta_i \, \delta_j (a_2) \, b_0) \, .
\end{eqnarray}
The resulting formula for $b_2$ is quite long and the next step is to perform the integration  $\alpha(\lambda)=\int b_2(\xi,\lambda)d^2\xi$. Note, before starting, that by homogeneity of symbols one has
 $$
 b_2(v\xi,v^2\lambda)=v^{-4}b_2(\xi,\lambda)
 $$
 and we thus know that $\alpha(\lambda)$ is homogeneous of degree $-1$ in $\lambda$ since
$$
 \int b_2(\xi,u\lambda)d^2\xi=u^{-2} \int b_2(u^{-1/2}\xi,\lambda)d^2\xi=
 u^{-1} \int b_2(\xi',\lambda)d^2\xi'\,.
 $$
 Moreover the next step in order to obtain the constant term in the heat expansion is to do an integral in the variable $\lambda$ of the form
 $$
\frac{1}{2\pi i}  \int_C  e^{-t\lambda} \alpha(\lambda) d\lambda
$$
 where the contour $C$ around the positive real axis is chosen in such a way that
 $$
 \frac{1}{2\pi i}  \int_C e^{-t\lambda}\frac{1}{s-\lambda}d\lambda=e^{-ts}\,.
 $$
 Applying this equality for $s=0$ one gets that
 $$
 \frac{1}{2\pi i}  \int_C  e^{-t\lambda} \alpha(\lambda) d\lambda=\alpha(-1)\,.
 $$
 Thus we can already simplify and fix $\lambda=-1$ before we perform the integration in $d^2\xi=d\xi_1d\xi_2$. To start this integration we follow \cite{FK} and perform the change of variables, using $\tau_1=\Re(\tau)$ and $\tau_2=\Im(\tau)$,
\begin{equation*}
    \xi_1=r \cos(\theta)-r\frac{ \tau_1}{\tau_2}\sin(\theta)\,, \ \ \xi_2=\frac{ r\,\sin(\theta)}{\tau_2}
\end{equation*}
The Jacobian of the change of coordinates is $\frac{r}{\tau_2}$. Moreover after changing variables one gets
$$
\xi_1^2+2\Re(\tau)\xi_1\xi_2+|\tau|^2\xi_2^2=r^2.
$$
One performs the integration in the angular variable $\theta$ first and the  terms one obtains  can be organized in the form
\begin{equation}\label{bigsum}
   - \frac{2\pi r}{\tau_2}\left(T_0+\tau_1 T_{1,0}+\tau_2 T_{0,1}+|\tau|^2 T_2       \right)b_0
\end{equation}
where one will note the overall multiplication on the right by $b_0$ which for our purpose cannot be moved in front as was done in \cite{Paula}.

The obtained terms are, except for this nuance, similar to those of \cite{Paula} and \cite{FK}, but
for the operator $\pertc$ there is a non-trivial term  $T_{0,1}$ which is given by

\bigskip

\begin{math}
T_{0,1}=i r^2 b_0 k \delta _1(k) b_0 k \delta _2(k)+i r^2 b_0 k \delta _1(k) b_0 \delta _2(k) k-i r^4 b_0 k \delta _1(k) k^2 b_0^2 k \delta _2(k)\\-i r^4 b_0 k \delta _1(k) k^2 b_0^2 \delta _2(k) k-i r^2 b_0 k \delta _2(k) b_0 k \delta _1(k)-i r^2 b_0 k \delta _2(k) b_0 \delta _1(k) k\\+i r^4 b_0 k \delta _2(k) k^2 b_0^2 k \delta _1(k)+i r^4 b_0 k \delta _2(k) k^2 b_0^2 \delta _1(k) k+i r^2 b_0 \delta _1(k) k b_0 k \delta _2(k)\\+i r^2 b_0 \delta _1(k) k b_0 \delta _2(k) k-i r^4 b_0 \delta _1(k) k k^2 b_0^2 k \delta _2(k)-i r^4 b_0 \delta _1(k) k k^2 b_0^2 \delta _2(k) k\\-i r^2 b_0 \delta _2(k) k b_0 k \delta _1(k)-i r^2 b_0 \delta _2(k) k b_0 \delta _1(k) k+i r^4 b_0 \delta _2(k) k k^2 b_0^2 k \delta _1(k)\\+i r^4 b_0 \delta _2(k) k k^2 b_0^2 \delta _1(k) k.
\end{math}

\bigskip

One finds that the terms $T_0$ and $T_2$ are equal and given by

\bigskip

\begin{math}
T_0=T_2=-2 r^2 k^2 b_0^2  k  \delta _1^2(k)-4 r^2 k^2 b_0^2  \delta _1(k)  \delta _1(k)-2 r^2 k^2 b_0^2  \delta _1^2(k)  k\\+2 r^4 k^4 b_0^3  k  \delta _1^2(k)+4 r^4 k^4 b_0^3  \delta _1(k)  \delta _1(k)+2 r^4 k^4 b_0^3  \delta _1^2(k)  k-r^2 b_0  k  \delta _1(k)  b_0  k  \delta _1(k)\\-r^2 b_0  k  \delta _1(k)  b_0  \delta _1(k)  k+r^4 b_0  k  \delta _1(k)  k^2 b_0^2  k  \delta _1(k)+r^4 b_0  k  \delta _1(k)  k^2 b_0^2  \delta _1(k)  k\\-r^2 b_0  \delta _1(k)  k  b_0  k  \delta _1(k)-r^2 b_0  \delta _1(k)  k  b_0  \delta _1(k)  k+r^4 b_0  \delta _1(k)  k  k^2 b_0^2  k  \delta _1(k)\\+r^4 b_0  \delta _1(k)  k  k^2 b_0^2  \delta _1(k)  k+6 r^4 k^2 b_0^2  k  \delta _1(k)  b_0  k  \delta _1(k)+6 r^4 k^2 b_0^2  k  \delta _1(k)  b_0  \delta _1(k)  k\\-2 r^6 k^2 b_0^2  k  \delta _1(k)  k^2 b_0^2  k  \delta _1(k)-2 r^6 k^2 b_0^2  k  \delta _1(k)  k^2 b_0^2  \delta _1(k)  k+6 r^4 k^2 b_0^2  \delta _1(k)  k  b_0  k  \delta _1(k)\\+6 r^4 k^2 b_0^2  \delta _1(k)  k  b_0  \delta _1(k)  k-2 r^6 k^2 b_0^2  \delta _1(k)  k  k^2 b_0^2  k  \delta _1(k)-2 r^6 k^2 b_0^2  \delta _1(k)  k  k^2 b_0^2  \delta _1(k)  k\\-4 r^6 k^4 b_0^3  k  \delta _1(k)  b_0  k  \delta _1(k)-4 r^6 k^4 b_0^3  k  \delta _1(k)  b_0  \delta _1(k)  k-4 r^6 k^4 b_0^3  \delta _1(k)  k  b_0  k  \delta _1(k)\\-4 r^6 k^4 b_0^3  \delta _1(k)  k  b_0  \delta _1(k)  k.
\end{math}

\bigskip

The term $T_{1,0}$ is more complicated and we give it for completeness, it is of the form
$$
T_{1,0}=r^2 T_{1,0}^{(2)}+r^4 T_{1,0}^{(4)}+ r^6 T_{1,0}^{(6)}
$$
where

\bigskip

\begin{math}
T_{1,0}^{(2)} =-4 k^2 b_0^2 k \delta _1 \delta _2(k)-4 k^2 b_0^2 \delta _1(k) \delta _2(k)-4 k^2 b_0^2 \delta _2(k) \delta _1(k)-4 k^2 b_0^2 \delta _1 \delta _2(k) k\\-b_0 k \delta _1(k) b_0 k \delta _2(k)-b_0 k \delta _1(k) b_0 \delta _2(k) k-b_0 k \delta _2(k) b_0 k \delta _1(k)-b_0 k \delta _2(k) b_0 \delta _1(k) k\\-b_0 \delta _1(k) k b_0 k \delta _2(k)-b_0 \delta _1(k) k b_0 \delta _2(k) k-b_0 \delta _2(k) k b_0 k \delta _1(k)-b_0 \delta _2(k) k b_0 \delta _1(k) k.
\end{math}

\bigskip
\begin{math}
T_{1,0}^{(4)} = 4 k^4 b_0^3 k \delta _1 \delta _2(k)+4 k^4 b_0^3 \delta _1(k) \delta _2(k)+4 k^4 b_0^3 \delta _2(k) \delta _1(k)+4 k^4 b_0^3 \delta _1 \delta _2(k) k\\+b_0 k \delta _1(k) k^2 b_0^2 k \delta _2(k)+b_0 k \delta _1(k) k^2 b_0^2 \delta _2(k) k+b_0 k \delta _2(k) k^2 b_0^2 k \delta _1(k)+b_0 k \delta _2(k) k^2 b_0^2 \delta _1(k) k\\+b_0 \delta _1(k) k k^2 b_0^2 k \delta _2(k)+b_0 \delta _1(k) k k^2 b_0^2 \delta _2(k) k+b_0 \delta _2(k) k k^2 b_0^2 k \delta _1(k)+b_0 \delta _2(k) k k^2 b_0^2 \delta _1(k) k\\+6 k^2 b_0^2 k \delta _1(k) b_0 k \delta _2(k)+6 k^2 b_0^2 k \delta _1(k) b_0 \delta _2(k) k+6 k^2 b_0^2 k \delta _2(k) b_0 k \delta _1(k)+6 k^2 b_0^2 k \delta _2(k) b_0 \delta _1(k) k\\+6 k^2 b_0^2 \delta _1(k) k b_0 k \delta _2(k)+6 k^2 b_0^2 \delta _1(k) k b_0 \delta _2(k) k+6 k^2 b_0^2 \delta _2(k) k b_0 k \delta _1(k)+6 k^2 b_0^2 \delta _2(k) k b_0 \delta _1(k) k.
\end{math}

\bigskip
\begin{math}
T_{1,0}^{(6)} = -2 k^2 b_0^2 k \delta _1(k) k^2 b_0^2 k \delta _2(k)-2 k^2 b_0^2 k \delta _1(k) k^2 b_0^2 \delta _2(k) k-2 k^2 b_0^2 k \delta _2(k) k^2 b_0^2 k \delta _1(k)\\-2 k^2 b_0^2 k \delta _2(k) k^2 b_0^2 \delta _1(k) k-2 k^2 b_0^2 \delta _1(k) k k^2 b_0^2 k \delta _2(k)-2 k^2 b_0^2 \delta _1(k) k k^2 b_0^2 \delta _2(k) k\\-2 k^2 b_0^2 \delta _2(k) k k^2 b_0^2 k \delta _1(k)-2 k^2 b_0^2 \delta _2(k) k k^2 b_0^2 \delta _1(k) k-4 k^4 b_0^3 k \delta _1(k) b_0 k \delta _2(k)\\-4 k^4 b_0^3 k \delta _1(k) b_0 \delta _2(k) k-4 k^4 b_0^3 k \delta _2(k) b_0 k \delta _1(k)-4 k^4 b_0^3 k \delta _2(k) b_0 \delta _1(k) k\\-4 k^4 b_0^3 \delta _1(k) k b_0 k \delta _2(k)-4 k^4 b_0^3 \delta _1(k) k b_0 \delta _2(k) k-4 k^4 b_0^3 \delta _2(k) k b_0 k \delta _1(k)\\-4 k^4 b_0^3 \delta _2(k) k b_0 \delta _1(k) k.
\end{math}

\bigskip

We check that the coefficient $T^{(j)}$ of $r^j$ in the term $T$ is non-zero only for even $j$
and that the total power of $b_0$ involved in $T^{(2s)}$ is $s+1$. Thus for $rTb_0$ as in \eqref{bigsum} we get that the general form is a sum
\begin{equation}\label{genterm}
    r T^{(2s)}b_0=\sum  b_0^{-m_0}  \rho_1
 b_0^{-m_1}\cdots  \rho_{\ell} b_0^{-m_\ell}r^{2(\sum m_j-2)+1}
\end{equation}
In order to put all these terms in a canonical form one moves the powers of $k$ to the left using the commutation of $k$ with $b_0$ and the rule
\begin{equation}\label{rulesk}
    a k^n=k^n \modu^{n/2}(a)\qqq a\in A_{\theta}^{\infty}.
\end{equation}
Thus for instance the first term of $T_{1,0}^{(6)}$ which is  $-2 k^2 b_0^2 k \delta _1(k) k^2 b_0^2 k \delta _2(k)$ is rewritten as
$$
-2 k^2 b_0^2 k \delta _1(k) k^2 b_0^2 k \delta _2(k)=
-2 k^6 b_0^2  \modu^{3/2}(\delta _1(k)) b_0^2  \delta _2(k)
$$
For such terms which are quadratic in the $\delta _j(k)$ we can use the simple formula
\begin{equation}\label{expansimple}
    \delta _j(k)=k f(\modu)\left(\delta_j(\log(k))\right)\,, \ \ f(u)=\frac{2 \left(-1+\sqrt{u}\right)}{\log(u)}
\end{equation}
which is justified below in \S \ref{expansect}. Thus the above term can be written as
$$
-2 k^6 b_0^2  \modu^{3/2}(\delta _1(k)) b_0^2  \delta _2(k)=-2 k^7 b_0^2
\modu^{2}(\delta _1(k)) b_0^2 f(\modu)\left(\delta_2(\log(k))\right)
$$
$$
=-2 k^8 b_0^2
\modu^{2}f(\modu)\left(\delta_1(\log(k))\right) b_0^2 f(\modu)\left(\delta_2(\log(k))\right)
$$
Besides terms which are quadratic in the $\delta _j(k)$ we also get terms which involve second derivatives of $k$. Thus for instance the first term of $T_{1,0}^{(2)}$ gives
$$
4 k^4 b_0^3 k \delta _1 \delta _2(k)=4 k^5 b_0^3  \delta _1 \delta _2(k)
$$
 For the terms which involve second derivatives of $k$ we need to carefully reexpress them in terms of second derivatives of $\log k$ as we now explain.

\subsection{Expansional} \hfill\medskip \label{expansect}

We write the modular automorphism in the form
\begin{equation*}
\modu (x) = e^{-h} x e^h=k^{-2}x\,k^2\,, \ \ k=e^{\frac h 2}
\end{equation*}
so that one has the permutation rule
\begin{equation*}
x\,k=k\,\modu^{\frac 12} (x)
\end{equation*}
The expansional formula can be written as
\begin{equation*}
    e^{A+B}=\sum_n\int_{\sum s_j=1, \, s_j\geq 0}e^{s_0 A}Be^{s_1 A}\ldots Be^{s_n A}\prod ds_j
\end{equation*}
We take $A=\log k$ and for $B$ the term one gets by expanding $\alpha_{t_1,t_2}(A)$ around $t_j=0$, \ie using purely imaginary arguments $t_j$,
\begin{equation*}
    B=t_1\delta_1(\log k)+t_2\delta_2(\log k)+\frac 12 t_1^2\delta_1^2(\log k)+\frac 12 t_2^2\delta_2^2(\log k)+ t_1 t_2\delta_1\delta_2(\log k)+\cdots
\end{equation*}
One has a similar expansion for $\alpha_{t_1,t_2}(k)$ which shows that \eg $\delta_1^2(k)$ is obtained from the coefficient of $\frac 12 t_1^2$ in $e^{A+B}$. More precisely one writes
$e^{A+B}$ as
$$
  e^{A+B}= e^A\left(1+\int_0^1 e^{(-1+s_0) A}Be^{(1-s_0) A}ds_0\right. $$ $$\left.+\int_0^1\int_0^{1-s_0} e^{(-1+s_0) A}Be^{((1-s_0)-(1-s_0-s_1)) A}Be^{(1-s_0-s_1)A}ds_1ds_0+\cdots\right)
$$
$$
=k\left(1+\int_0^1 \modu^{\frac u2}(B) du+\int_0^1\int_0^{u}\modu^{\frac u2}(B)\modu^{\frac v2}(B)dv du+\cdots\right)
$$
where in the second integral one lets $u=1-s_0$ which varies from $0$ to $1$ and $v=1-s_0-s_1$ which varies from $0$ to $u$. In terms of the derivations $\nabla_j=\log \modu^{(j)}$ this gives the formula
\begin{equation}\label{nablaj}
 k^{-1} \delta_1^2(k)=f(\nabla)\delta_1^2(\log k)+2 g(\nabla_1,\nabla_2)\delta_1(\log k)\delta_1(\log k)
\end{equation}
where one has
\begin{equation}\label{nablaj1}
f(s)=\int_0^1e^{us/2}du=\frac{2 \left(-1+e^{s/2}\right)}{s}
\end{equation}
and
\begin{equation}\label{nablaj2}
g(s,t)=\int_0^1\int_0^{u}e^{us/2}e^{vt/2}dv du=\frac{4 \left(e^{\frac{s+t}{2}}
s+t-e^{s/2} (s+t)\right)}{s t (s+t)}
\end{equation}
Note the coefficient $2$ in front of $g$ since $\delta_1^2(k)$ is obtained from the coefficient of
$\frac 12 t_1^2$ in $e^{A+B}$.

\subsection{Rearrangement Lemma} \hfill\medskip\label{sectrearang}

In order to perform the integration in the radial variable $r$ one needs a more general lemma than Lemma 6.2 of \cite{Paula}. Note that only even powers of $r$ appear in the expressions for the terms $T$ but since one needs to multiply by the Jacobian of the change of coordinates, the integration in $r$ only involves odd powers of $r$. Thus it is natural to let $u=r^2$ so that
$du=2rdr$.

\begin{lem}\label{permlem}
-- For every element $\rho_j$ of $A_{\theta}^{\infty}$ and every integers $m_j>0$ one has,
\begin{eqnarray}\label{permid}
   && \int_0^{\infty} ( k^2 \, u + 1)^{-m_0}  \rho_1
 ( k^2 \, u + 1)^{-m_1}\cdots  \rho_{\ell} ( k^2 \, u + 1)^{-m_\ell}u^{\sum m_j-2} du\nonumber \\
  &=& k^{-2(\sum m_j-1)}\,F_{m_0,m_1,\ldots,m_\ell}
  (\modu_{(1)}, \modu_{(2)},\ldots,\modu_{(\ell)}) (\rho_1\rho_2 \ldots \rho_{\ell}) \,,
\end{eqnarray}
where the  function $F_{m_0,m_1,\ldots,m_\ell}(u_1,u_2,\ldots,u_{\ell})$ is
\begin{eqnarray}\label{fctF}
&& \quad \ F_{m_0,m_1,\ldots,m_\ell}(u_1,u_2,\ldots,u_{\ell})\nonumber \\
  &=& \int_0^{\infty}(u + 1)^{-m_0} \prod_1^\ell (u\prod_1^j u_h+1)^{-m_j}u^{\sum m_j-2} du
\end{eqnarray}
and $\modu_{(i)}$ signifies that $\modu$ acts on the $i$-th factor.
\end{lem}

\proof Let $G_n$ be the inverse Fourier transform of the function
$$
t\mapsto (e^{t/2}+e^{-t/2})^{-n}
$$
One has
$$
G_1(s)=\frac{1}{e^{-\pi  s}+e^{\pi  s}}\, \ G_2(s)=\frac{s}{-e^{-\pi  s}+e^{\pi  s}}\,, \ G_3(s)=
\frac{1+4 s^2}{8 \left(e^{-\pi  s}+e^{\pi  s}\right)}
$$
$$
G_4(s)=\frac{s \left(1+s^2\right)}{6 \left(-e^{-\pi  s}+e^{\pi  s}\right)}\, , \ G_5(s)=\frac{9+40 s^2+16 s^4}{384 \left(e^{-\pi  s}+e^{\pi  s}\right)}
$$
and more generally
$$
G_n(s)=\frac{P_n(s)}{e^{\pi  s}-(-1)^n e^{-\pi  s}}
$$
The role of the polynomials $P_n$ which appear in the numerator is to
compensate for the zeros of the denominator in larger and larger
strips. Thus the imaginary part of the first singularity of $G_n$ is $\frac n2$. The inverse Fourier transform of the function, defined for $\alpha\in ]0,n[$ by
$$
H_{n,\alpha}(t)=e^{(n-\alpha)t}(e^{t}+1)^{-n}
$$
is $G_{n,\alpha}(s)=G_n(s-i(\frac n2-\alpha))$ so that
\begin{equation}\label{fourierhn}
H_{n,\alpha}(t)=\int_{-\infty}^\infty G_n(s-i(\frac n2-\alpha))e^{-ist}ds
\end{equation}
We now perform in the left hand side of \eqref{permid} the change of variables $u=e^s$, with $k=e^{f/2}$, and obtain
$$
J=\int_{-\infty}^{\infty} ( e^{(s+f)} + 1)^{-m_0}  \rho_1
 (e^{(s+f)} + 1)^{-m_1}\cdots  \rho_{\ell} ( e^{(s+f)} + 1)^{-m_\ell}e^{(\sum m_j-1)s} ds
 $$
We now choose positive real numbers $\alpha_j>0$ such that $\sum \alpha_j=1$ and replace each term
$(e^{(s+f)} + 1)^{-m_j}$ by $e^{(m_j-\alpha_j)(s+f)}(e^{(s+f)} + 1)^{-m_j}$. This is fine for the
$s$ variable since it accounts for  the term $e^{(\sum m_j-1)s}$, but taking care of  the $\rho_j$ one gets
$$
J=e^{-(\sum m_j-1)f}\int_{-\infty}^{\infty}H_{m_0,\alpha_0}(s+f)\modu^{\beta_1}(\rho_1)
H_{m_1,\alpha_1}(s+f)\cdots\modu^{\beta_\ell}(\rho_\ell)H_{m_\ell,\alpha_\ell}(s+f)ds
$$
where
\begin{equation}\label{betaj}
\beta_j=-\sum_j^\ell(m_i-\alpha_i)
\end{equation}
We set $\rho'_j=\modu^{\beta_j}(\rho_j)$. Using \eqref{fourierhn} we can then write $J$ as an integral
of terms of the form
\begin{equation}\label{gentermbis}
e^{-(\sum m_j-1)f}H_{m_0,\alpha_0}(s+f)\rho'_1 e^{-i(s+f)t_1}\rho'_2\cdots e^{-i(s+f)t_{\ell-1}}\rho'_\ell e^{-i(s+f)t_{\ell}}
\end{equation}
with respect to the measure given by
$$
\prod_1^\ell G_{m_j,\alpha_j}(t_j)dt_j ds
$$
The term \eqref{genterm} can be written as
$$
e^{-(\sum m_j-1)f}H_{m_0,\alpha_0}(s+f)e^{-i(\sum_1^\ell t_j)(s+f)}\prod\modu^{-i\sum_h^\ell t_j}(\rho'_h ) $$
One has
$$
\int_{-\infty}^{\infty} H_{m_0,\alpha_0}(s+f)e^{-i(\sum_1^\ell t_j)(s+f)}ds=2\pi G_{m_0,\alpha_0}(-\sum_1^\ell t_j)
$$
Thus one obtains
$$
J=2\pi e^{-(\sum m_j-1)f}\int\prod\modu^{-i\sum_h^\ell t_j}(\rho'_h )G_{m_0,\alpha_0}(-\sum_1^\ell t_j)\prod_1^\ell G_{m_j,\alpha_j}(t_j)dt_j
$$
We now replace $\rho'_j=\modu^{\beta_j}(\rho_j)$ and replace the term
$$
\modu^{-i\sum_h^\ell t_j}(\rho'_h )=\modu^{-i\sum_h^\ell t_j+\beta_h}(\rho_h )
$$
by
$$
u_h^{-i\sum_h^\ell t_j+\beta_h}
$$
and we are dealing with the scalar function of $\ell$ variables
\begin{eqnarray}\label{fctF1}
   && F_{m_0,m_1,\ldots,m_\ell}(u_1,u_2,\ldots,u_{\ell})\\
 &=& 2\pi \int\prod u_h^{-i\sum_h^\ell t_j+\beta_h}G_{m_0,\alpha_0}(-\sum_1^\ell t_j)\prod_1^\ell G_{m_j,\alpha_j}(t_j)dt_j \nonumber
\end{eqnarray}
We can now write
$$
2\pi G_{m_0,\alpha_0}(-\sum_1^\ell t_j)=\int_{-\infty}^{\infty} H_{m_0,\alpha_0}(s)e^{-i(\sum_1^\ell t_j)s}ds
$$
With $u_h=e^{s_h}$ we can perform the integral in $t_j$ one gets that the coefficient of $t_j$ in the exponent is $-is-i\sum_1^j s_h$ so that the integral in $t_j$ gives the Fourier transform of $G_{m_j,\alpha_j}$ at $s+ \sum_1^j s_h$. This is
$$
e^{(m_j-\alpha_j)(s+ \sum_1^j s_h)}(e^{s+ \sum_1^j s_h}+1)^{-m_j}=e^{(m_j-\alpha_j)s}(\prod_1^j u_h)^{(m_j-\alpha_j)}(e^s\prod_1^j u_h+1)^{-m_j}
$$
In the product of these terms from $j=1$ to $j=\ell$ one gets $u_h$ with the exponent
$\sum_h^\ell (m_j-\alpha_j)$. Thus this cancels the term $u_h^{\beta_h}$. We thus get
\begin{eqnarray}\label{fctF2}
   && F_{m_0,m_1,\ldots,m_\ell}(u_1,u_2,\ldots,u_{\ell})\\
 &=& \int_{-\infty}^{\infty}(e^s + 1)^{-m_0} \prod_1^\ell (e^s\prod_1^j u_h+1)^{-m_j}e^{(\sum m_j-1)s} ds \nonumber
\end{eqnarray}
which proves the required equality. \endproof

One has by construction
$$
F_{m_0,m_1,\ldots,m_\ell}(u_1,u_2,\ldots,u_{\ell})=H_{m_0,m_1,\ldots,m_\ell}(u_1,u_1 u_2,\ldots,u_1\cdots u_{\ell})
$$
where
\begin{eqnarray}\label{fctH}
&& \quad \ H_{m_0,m_1,\ldots,m_\ell}(u_1,u_2,\ldots,u_{\ell})\nonumber \\
  &=& \int_0^{\infty}(u + 1)^{-m_0} \prod_1^\ell (u  u_h+1)^{-m_j}u^{\sum m_j-2} du
\end{eqnarray}
The first few functions of two variables that we shall use are given as follows\begin{tiny}
\begin{eqnarray}
 H_{1,1,1}(a,b) &=& \frac{(-1+b) \text{Log}(a)-(-1+a) \text{Log}(b)}{(-1+a) (-1+b) (-a+b)} \nonumber\\
  H_{1,2,1}(a,b) &=& \frac{(-1+b) ((-1+a) (a-b)+a (1-2 a+b) \text{Log}(a))+(-1+a)^2 a \text{Log}(b)}{(-1+a)^2 a (a-b)^2 (-1+b)}\nonumber\\
  H_{2,1,1}(a,b) &=& \frac{(-1+b)^2 \text{Log}(a)+(-1+a) ((a-b) (-1+b)-(-1+a) \text{Log}(b))}{(-1+a)^2 (a-b) (-1+b)^2}\nonumber \\
  H_{2,2,1}(a,b)&=& \frac{(-1+b) \left((-1+a) (a-b) \left(1+a^2-(1+a) b\right)+a (-1+3 a-2 b) (-1+b) \text{Log}(a)\right)-(-1+a)^3 a \text{Log}(b)}{(-1+a)^3 a (a-b)^2 (-1+b)^2}\nonumber \\
  H_{3,1,1}(a,b) &=& \frac{(-1+a) (5+a (-3+b)-3 b) (a-b) (-1+b)-2 (-1+b)^3 \text{Log}(a)+2 (-1+a)^3 \text{Log}(b)}{2 (-1+a)^3 (a-b) (-1+b)^3}\nonumber
\end{eqnarray}
\end{tiny}

\end{document}